\documentclass[a4paper, 12pt]{amsart}
\usepackage{amsfonts, amsthm, amssymb, amsmath}
\usepackage{mathrsfs,array}
\usepackage{xy}
\usepackage{hyperref}
\usepackage{verbatim}
\usepackage[latin1]{inputenc}
\input xy
\xyoption{all}

\usepackage[titletoc]{appendix}
\usepackage{color}

\newtheorem*{theoA}{Theorem A}
\newtheorem*{theoB}{Theorem B}

\newtheorem{theo}{Theorem}[section]
\newtheorem{prop}[theo]{Proposition}
\newtheorem{defi}[theo]{Definition}
\newtheorem{lemm}[theo]{Lemma}
\newtheorem{coro}[theo]{Corollary}

\theoremstyle{definition}
\newtheorem{rema}[theo]{Remark}
\newtheorem{exam}[theo]{Example}

\newcommand{\mc}{\mathcal}
\newcommand{\mf}{\mathfrak}
\newcommand{\mbf}{\mathbf}

\newcommand{\ra}{\rightarrow}
\newcommand{\mrm}{\mathrm}

\DeclareMathOperator{\Res}{Res}
\DeclareMathOperator{\Ad}{Ad}
\DeclareMathOperator{\Ext}{Ext}

\DeclareMathOperator{\Hom}{Hom}

\DeclareMathOperator{\Ind}{Ind}

\DeclareMathOperator{\GL}{GL}

\DeclareMathOperator{\soc}{soc}

\DeclareMathOperator{\an}{an}

\DeclareMathOperator{\sm}{sm}
\DeclareMathOperator{\bal}{bal}
\DeclareMathOperator{\alg}{alg}

\DeclareMathOperator{\lp}{lp}

\DeclareMathOperator{\pol}{pol}

\DeclareMathOperator{\ad}{ad}

\title{An adjunction formula for the Emerton--Jacquet functor}
\date {\today}
\author{John Bergdall and Przemys\l aw Chojecki}
\address{John Bergdall\\Department of Mathematics and Statistics \\ Boston University \\ 111 Cummington Mall \\ Boston, MA 02215\\USA}
\email{bergdall@math.bu.edu}
\urladdr{http://math.bu.edu/people/bergdall}

\address{Przemys\l aw Chojecki\\Mathematical Institute\\
University of Oxford\\
Andrew Wiles Building\\
Radcliffe Observatory Quarter\\
Woodstock Road\\
Oxford OX2 6GG\\
England}
\email{chojecki@maths.ox.ac.uk}
\urladdr{http://people.maths.ox.ac.uk/chojecki/}

\thanks{The authors would like to thank David Hansen, Keenan Kidwell, Sascha Orlik, Matthias Strauch and an anonymous referee for helpful comments and suggestions. We especially M.\ Strauch who kindly pointed out an embarrassing error in an earlier version of this paper and for discussing a number of revisions with the first author. Some of this work was carried out while one of us (J.B.) visited the other (P.C.) in Oxford. It is our pleasure to thank the mathematics department of the University of Oxford for extraordinary hospitality during this visit. The first-named author's research was partially supported by NSF award DMS-1402005. The second-named author's research was partially funded by EPSRC grant EP/L005190/1}


\begin{document}

\begin{abstract}
The Emerton--Jacquet functor is a tool for studying locally analytic representations of $p$-adic Lie groups. It provides a way to access the theory of $p$-adic automorphic forms. Here we give an adjunction formula for the Emerton--Jacquet functor, relating it directly to locally analytic inductions, under a strict hypothesis that we call {\em non-critical}. We also further study the relationship to socles of principal series in the non-critical setting.
\end{abstract}

\maketitle

\addtocontents{toc}{\protect\setcounter{tocdepth}{1}}
\tableofcontents

\section{Introduction}
Let $p$ be a prime. Throughout this paper we fix a finite extension $K/\mbf Q_p$ and let $\mbf G$ be a connected, reductive and split algebraic group defined over $K$. We fix a maximal split torus $\mbf T$ and a Borel subgroup $\mbf B$ containing $\mbf T$. We also let $\mbf P$ denote a standard parabolic containing $\mbf B$. The parabolic subgroup $\mbf P$ has a Levi decomposition $\mbf P = \mbf N_{\mbf P} \mbf L_{\mbf P}$, with $\mbf N_{\mbf P}$ being the maximal unipotent subgroup of $\mbf P$. We denote the opposite subgroup of $\mbf{P}$ by $\mbf P^-$. If $\mbf H$ is an algebraic group defined over $K$ we use the Roman letters $H := \mbf H(K)$ to denote the corresponding $p$-adic Lie group which we consider as a locally $\mathbf Q_p$-analytic group (not as a locally $K$-analytic group). Let $L$ be another finite extension of $\mbf Q_p$, which we will use as a field of coefficients.

\subsection{The Emerton-Jacquet functor}
If $V$ is a smooth and admissible $L$-linear representation of $G$ then one can associate a smooth and admissible $L$-linear representation $J_P(V):= V_{N_P}$ of the Levi factor $L_P$, called the Jacquet module of $V$. The functor $J_P(-)$ is exact, and it turns out that the irreducible constituents of $J_P(V)$ give rise to  irreducible constituents of $V$ via adjunction with smooth parabolic induction $\Ind_{P^-}^G(-)^{\sm}$. Specifically, if $U$ is a smooth and admissible representation of $L_P$, seen as a representation of $P^{-}$ via inflation, then there is a natural isomorphism
\begin{equation}\label{eqn:old-jacquet-adjunction}
\Hom_{G}(\Ind_{P^-}^G(U)^{\sm}, V) \overset{\simeq}{\longrightarrow} \Hom_{L_P}(U(\delta_P), J_P(V)),
\end{equation}
where $\delta_P$ is the modulus character of $P$.

In \cite{em2}, Emerton extended the functor $J_P(-)$ to certain categories of locally analytic representations of $G$ on $L$-vector spaces (throughout the paper, locally analytic refers to locally $\mathbf Q_p$-analytic). We will call this extended functor the Emerton--Jacquet functor but still denote it by $J_P(-)$. If $U$ is a suitable locally analytic representation of $L_P$ and $V$ is a suitable locally analytic representation of $G$, one could ask for an adjunction formula in the spirit of \eqref{eqn:old-jacquet-adjunction}, relating $\Hom_{L_P}(U(\delta_P), J_P(V))$ to $\Hom_G(\Ind_{P^-}^G(U)^{\an}, V)$. Here, $\Ind_{P^-}^G(-)^{\an}$ is the locally analytic induction and suitable refers to, for example, the hypotheses in the introduction of \cite{em3}. The $\Hom$-spaces are meant to be continuous morphism (a convention that remains throughout the paper).

It was explained by Emerton that the na\"ive generalisation of \eqref{eqn:old-jacquet-adjunction} is not generally correct. Indeed, the main result of Emerton's paper \cite{em3} is an isomorphism
\begin{equation}\label{eqn:new-adjunction}
\Hom_G(I_{P^-}^G(U), V) \overset{\simeq}{\longrightarrow} \Hom_{L_P}(U(\delta_P),J_P(V))^{\bal}
\end{equation}
where $I_{P^-}^G(U)$ is a certain {\em subrepresentation} of $\Ind_{P^-}^G(U)^{\an}$ and $\Hom_{L_P}(U(\delta_P),J_P(V))^{\bal}$ is the ``balanced'' subspace of $\Hom_{L_P}(U(\delta_P),J_P(V))$ (see \cite[Definition 0.8]{em3}). 

Breuil showed in \cite{br} that one can remove the balanced condition on the right-hand side of  Emerton's formula \eqref{eqn:new-adjunction} at the expense of replacing the locally analytic induction $\Ind_{P^-}^G(U)^{\an}$, or its subrepresentation $I_{P^-}^G(U)$, with a closely related locally analytic representation defined and studied by Orlik and Strauch \cite{os-old} (we recall work of Orlik--Strauch and Breuil in Section \ref{section:socle}).

\subsection{Statement of theorem}

Our main goal is to give sufficient practical conditions under which a na\"ive adjunction formula holds. Let us state our main theorem, and comment on the hypotheses (especially the notion of non-critical) afterwards. Below German fraktur letters refer to Lie algebras. See Section \ref{subsec:notations} for all the notations.

\begin{theoA}[Theorem \ref{theo:adj-theo}]\label{theoremA}
Suppose that $V$ is a very strongly admissible locally analytic representation of $G$ which is f-$\mf p$-acyclic, $U$ is a finite-dimensional locally analytic representation of $L_P$ which is irreducible as a module over $\mrm U({\mf l_{P,L}})$ and a sum of characters when restricted to $\mf t_L$, and $\pi$ is a finite length smooth representation of $L_P$ which admits a central character. Then, if $(U,\pi)$ is non-critical with respect to $V$, there is a canonical isomorphism
\begin{equation*}
\Hom_G(\Ind_{P^-}^G(U\otimes \pi(\delta_P^{-1}))^{\an}, V) \overset{\simeq}{\longrightarrow} \Hom_{L_P}(U\otimes \pi, J_P(V)).
\end{equation*}
\end{theoA}
The first hypothesis, that $V$ is very strongly admissible, is a relatively natural one. For example, it appears in \cite[Theorem 0.13]{em3}. We refer to Definition \ref{defi:vsa} for this and the definition of f-$\mathfrak p$-acyclic. See Example \ref{example:acyclic} for an example (the localisation of $p$-adically completed and compactly supported cohomology of modular curves at non-Eisenstein ideals is also an example; see \cite[Corollarie 5.1.3]{be}).

The hypotheses on $U$ and $\pi$ taken individually should be self-explanatory. The crucial hypothesis in Theorem A then is that the pair $(U,\pi)$ be {\em non-critical} with respect to $V$. We note immediately that there is a sufficient condition, the condition of having {\em non-critical slope}, which depends only on $U$ and $\pi$ (see Remark \ref{rema:non-critical-slope}) but the definition of non-critical is more general and depends on $V$. We will give a brief explanation here and refer to Definition \ref{defi:badness} for more details.

Suppose for the moment that $\mbf P = \mbf B$ is a Borel subgroup so that $\mbf L_{\mbf P} = \mbf T$ is the torus. The hypothesis on $U$ in Theorem A implies that $U$ is a locally analytic character $\chi$ of $T$, and we assume that $\pi$ is trivial (by absorbing $\pi$ into $\chi$). If $\chi'$ is a locally analytic character of $T$ then there is the notion of $\chi'$ being strongly linked to $\chi$ (see Definition \ref{defi:strongly-linked}) which generalizes the well-known notion of strongly linked weights coming from the representation theory of the Lie algebra of $G$ (see \cite{hum}). The Emerton--Jacquet module $J_P(V)$ for $V$ as in Theorem A has a locally analytic action of $T$.  We say that the character $\chi$ (or the pair $(\chi,1)$) is non-critical with respect to $V$ if $J_P(V)^{T=\chi'} = (0)$ for every character $\chi' \neq \chi$ strongly linked to $\chi$. The definition is inspired by the definition of ``not bad'' that appears in  \cite{be} and \cite{bcho}.

If $\mbf P$ is not a Borel subgroup, the hypothesis on $U$ still allows us to obtain a character $\chi$ of $T$, a ``highest weight'' for $U$, which we may  restrict to the center $Z(L_P)$ of the Levi factor. The definition of non-critical is then phrased in terms of the eigenspaces of $J_P(V)$ under the action of the center $Z(L_P)$. The subtlety of which characters to qualify over is already present in the intricacy of so-called {\em generalized} Verma modules \cite[Section 9]{hum}. However, we suggest in Section \ref{subsec:motivation} why one should consider characters of the center, and one possible strategy to checking the non-critical hypothesis.

\subsection{Relationship with locally analytic socles}
Theorem A is meant to capture a strong relationship between representations appearing in an Emerton--Jacquet module $J_P(V)$ and principal series appearing in $V$. However, in contrast with the classical theory of smooth and admissible representations, the principal series $\Ind_{P^-}^G(U\otimes \pi)^{\an}$ may be highly reducible. For example, it is easy to construct examples of $U,\pi$ and $V$ such that every non-zero map $\Ind_{P^-}^G(U\otimes \pi)^{\an} \ra V$ factors through a proper quotient. And so, for Theorem A to be useful, one needs to prove a stronger version. For that, we restrict to the case where $U$ is an irreducible finite-dimensional {\em algebraic} representation of $L_P$, i.e. the induced representation of $L_P$ on an irreducible finite-dimensional algebraic representation of the underlying algebraic group $\Res_{K/\mathbf Q_p}\mbf L_P$.
\begin{theoB}[Theorem \ref{coro:socle}]
Suppose that $U$ is an irreducible finite-dimensional algebraic representation of $L_P$ and $\pi$ is a finite length smooth representation of $L_P$ admitting a central character, such that $\Ind_{P^-}^G(\pi)^{\sm}$ is irreducible. Let $V$ be a very strongly admissible, f-$\mf p$-acyclic representation of $G$ such that $(U,\pi)$ is non-critical with respect to $V$. Then the containment $\soc_G \Ind_{P^-}^G(U\otimes \pi(\delta_P^{-1}))^{\an} \subset \Ind_{P^-}^G(U\otimes \pi(\delta_P^{-1}))^{\an}$ induces a natural isomorphism
\begin{equation*}
\Hom _G(\Ind _{P^-} ^G(U \otimes \pi(\delta_P^{-1})) ^{\an}, V) \simeq \Hom _G (\soc _G \Ind _{P^-} ^G(U \otimes \pi(\delta_P^{-1})) ^{\an}, V)
\end{equation*}
\end{theoB}
Here, the notation $\soc_G(-)$ refers to the locally analytic {\em socle}, i.e. the sum of the topologically irreducible subrepresentations (under these assumptions, it is actually irreducible). The hypothesis on $\pi$ is sufficient, but not necessary (see the statement of Theorem \ref{coro:socle} and Remark \ref{rema:irreducibility}). The socles of principal series play a central role in recent conjectures of Breuil \cite{br}. Combining Theorems A and B results in an obvious corollary, which we omit.

\subsection{Methods}
The proof of Theorem A comes in two steps. The first step is to understand composition series of certain $(\mf g, P)$-modules within the category $\mc O^P$ recently introduced by Orlik and Strauch \cite{os}. The composition series we consider are clear generalizations of composition series of classical Verma modules. This is carried out in Section \ref{sec:verma-mods-fd} and constitutes a significant portion of our work (see Section \ref{subsec:verma2} for the results). It should also be of independent interest to have information on the composition series in $\mathcal O^P$.

The second step is to use the description of the composition series  to prove Theorem A. This is done in Section \ref{sec:adj-formula}. The crux of the argument follows from the $p$-adic functional analysis studied by Emerton in his two papers on the Emerton--Jacquet functor \cite{em2,em3}.

The proof of Theorem B, which we give in Section \ref{section:socle}, naturally falls out of the techniques we use to prove Theorem A, together with work of Orlik and Strauch \cite{os-old, os}, and Breuil \cite{br1}, on locally analytic principal series. A suitable generalization of some results in \cite{br1} would possibly remove the algebraic hypothesis on $U$. see Remark \ref{rema:irreducible}.

Breuil also proved an adjunction formula \cite[Th\'eor\`eme 4.3]{br} without any non-critical hypotheses, but replacing the locally analytic induction with the Orlik--Strauch representations discussed in Section \ref{section:socle}. A generalization of this formula to $\mathcal O^P$ would be interesting, and any application of such a formula would certainly require the results in Section \ref{sec:verma-mods-fd}.

\subsection{Global motivation}\label{subsec:motivation}
From a purely representation-theoretic point of view, one might wonder why the definition of non-critical qualifies over characters of the center, rather than representations of the Levi subgroup. In order to explain this, let us finish the introduction by placing our theorem and its hypotheses within the still emerging $p$-adic Langlands program. 

A rich source of locally analytic representations of $p$-adic Lie groups comes from the theory of $p$-adically completed cohomology. See \cite{ce} for a summary and further references. In the theory, one constructs $p$-adic Banach spaces, typically denoted by $\widehat H^i$, arising from the $p$-adically completed cohomology of locally symmetric spaces for an adelic group which locally at a place above $p$ is the $p$-adic Lie group $G$. The completion process equips $\widehat{H}^i$ with a natural continuous action of $G$. Passing to locally analytic vectors $\widehat H^i_{\an}$, we get a very strongly admissible locally analytic representation of $G$. The cohomology $\widehat{H}^i$ also has an auxiliary action of a Hecke algebra which we denote by $\mc H$. After localizing at a so-called non-Eisenstein maximal ideal $\mf m \subset \mc H$, we get a representation $V = \widehat H^i_{\an,\mf m}$ which will satisfy the hypotheses of Theorem A (in practice---the localization here is meant to force the cohomology to be acyclic for the derived action of the Lie algebra).

Already in the case of degree zero cohomology, there has been a significant amount of energy spent predicting the principal series representations of $G$ which should appear in spaces $\widehat H^0_{\an,\mf m}$ attached to certain definite unitary groups. See \cite[Section 4]{bh} and \cite[Section 6]{br}. More precisely, \cite[Conjectures 6.1 and 6.2]{br} predicts the principal series whose {\em socle} appears as a subrepresentation of certain cohomology spaces and \cite[Conjecture 4.2.2]{bh} gives a prediction for which continuous principal series (which arise by unitary completion from our setting) should appear. When the non-critical hypothesis is satisfied, Theorems A and B illuminates those conjectures.

The remaining link between Theorem A and $p$-adic Langlands is provided by the theory of {\em eigenvarieties}. Emerton showed in \cite{em1}, see also the work of Loeffler and Hill \cite{hl}, how to use the Emerton--Jacquet module $J_P(\widehat H^i_{\an})$  to construct a $p$-adic rigid analytic space $\mc E$ called an eigenvariety. The eigenvariety $\mc E$ parameterizes pairs $(\chi,\mf m)$, where $\chi$ is a locally analytic character of the center $Z(L_P)$ of the Levi and $\mf m \subset \mc H$ is a maximal ideal, such that the $\chi$-eigenspace for the action of $T$ on $J_P(\widehat H^i_{\an,\mf m})$ is non-zero. The eigenvarieties constructed this way vastly generalize the $p$-adic {\em eigencurve} constructed by Coleman and Mazur \cite{cm}. 

Taking $V = \widehat H^i_{\an,\mf m}$, the obstruction to applying Theorem A to $V$ can be reduced to the existence or non-existence of certain points on the eigenvariety $\mc E$. This final question belongs in the theory of {\em $p$-adic companion forms}, a topic studied previously by the authors in \cite{bcho}, which provides a test case for the motivation laid out here. For more on constructing companion forms, see \cite{bhs,ding}.

\subsection{Notations}\label{subsec:notations}
We summarize our notations for reference. We fix an algebraic closure $\overline{\mbf Q}_p$ and a finite extension $L/\mbf Q_p$ contained in $\overline{\mbf Q}_p$ (our field of coefficients). Let $K$ be another finite extension of $\mbf Q_p$ (our base field). Write $\mc{S}$ for the set of all $\mathbf Q_p$-linear embeddings $\sigma: K \ra \overline{\mbf Q}_p$. We assume throughout that $\sigma(K) \subset L$ for all $\sigma \in \mc S$. We then have a natural equality $K \otimes _{\mbf{Q}_p} L \simeq \prod _{\sigma \in \mc{S}} L_{\sigma}$, via $x \otimes y \mapsto (\sigma(x))$, of $K$-algebras, where $L_{\sigma}$ is the $K$-algebra $\sigma: K \rightarrow L$.  

If $V$ is a $\mbf Q_p$-vector space then we define $V_L := V\otimes_{\mbf Q_p} L$. If $V$ is a $K$-vector space, we can view it as a $\mbf{Q}_p$-vector space and we get $V _L = \prod_{\sigma \in \mc S} V_{\sigma}$ where $V_{\sigma} := V \otimes _{K} L_{\sigma}$. If $V$ is an $L$-vector space we use $V^{\ast}$ to denote the $L$-linear dual. 

If $\mbf H$ is an algebraic group over $K$ then we use a German letter $\mf h$ for its Lie algebra. The notation $\mrm U(\mf h)$ refers to the universal enveloping algebra of $\mf h$. We write $\Res_{K/\mbf Q_p} \mbf H$ for the restriction of scalars of $\mbf H$ down to $\mbf Q_p$. The Lie algebra of $\Res_{K/\mbf Q_p} \mbf H$ is naturally identified with $\mf h$. We also write $\mbf H_{/L}$ for the product $\mbf H_{/L} := (\Res_{K/\mbf Q_p} \mbf H) \times_{\mbf Q_p} L$. This is an algebraic group over $L$ now; its Lie algebra is $\mf h_L = \mf h \otimes_{\mbf Q_p} L$ and $\mbf H_{/L}$ decomposes into a finite product $\mbf H_{/L} \simeq \prod_{\sigma \in \mc S} \mbf H \times_{K} L_{\sigma}$.

We always write $\mbf G$ for a split connected reductive group defined over $K$. We fix a maximal split torus $\mbf T$ and write $\mbf B$ for the choice of a Borel subgroup containing $\mbf T$. We use $\mbf P$ for a standard parabolic subgroup with respect to $\mbf B$. If $\mbf P$ is a standard parabolic then we write $\mbf P = \mbf L_{\mbf P} \mbf N_{\mbf P}$ for its Levi decomposition, $\mbf N_{\mbf P}$ being the maximal unipotent subgroup of $\mbf P$ and $\mbf L_{\mbf P}$ being the correspondng Levi factor. We use the evident notations $\mf g, \mf p$, etc. for the Lie algebras of $\mbf G, \mbf P$, etc. The Lie algebras of $\mbf L_{\mbf P}$, resp. $\mbf N_{\mbf P}$, is written $\mf l_{P}$, resp. $\mf n_P$. Write $\mbf B^{-}$ for the Borel subgroup of $\mbf G$ opposite to $\mbf B$. Then if $\mbf P$ is a standard parabolic we write $\mbf P^{-} \supset \mbf B^-$ for its opposite. The group $\mbf G_{/L}$ is also a split connected reductive group over $L$ with maximal split torus $\mbf T_{/L}$ contained in the Borel subgroup $\mbf B_{/L}$, etc. The Lie algebra of ${\mbf L_{\mbf P}}_{/L}$, resp. ${\mbf N_{\mbf P}}_{/L}$, will be written $\mf l_{P,L}$, resp. $\mf n_{P,L}$.

We use Roman letters $G, P$, etc. for the $K$-points $\mbf G(K)$, $\mbf P(K)$. We view these as locally $\mbf Q_p$-analytic groups. Thus we may study locally $\mbf Q_p$-analytic representations of $G, P$, etc. on $L$-vector spaces. We say plainly locally analytic to mean locally $\mbf Q_p$-analytic; the context should always make this clear.

If $V$ is an $L$-linear locally analytic representation of $G$ then there is a natural induced $\mbf Q_p$-linear action of $\mf g$ on $V$, which extends to the structure of a left module over $\mrm U(\mf g_L)$. An $L$-linear $\mbf Q_p$-analytic representation of $G$ is said to be $\mbf Q_p$-algebraic if it is the induced action of $G \subset \mbf G_{/L}(L)$ on an algebraic representation $W$ of $\mbf G_{/L}$.

Let $X(\mbf T)$ be the abelian group of algebraic characters $\mbf T$ and $Y(\mbf T)$ be the abelian group of co-characters. We let $\Phi \subset X(\mbf T)$ be the set of roots of $\mbf T$ and $\Phi^+$ be the subset of $\Phi$ consisting of roots positive with respect to the choice of Borel $\mbf B$. If $\alpha \in \Phi$ then we write $\alpha^\vee \in Y(\mbf T)$ for the corresponding co-root. Let $\Phi^+ \subset \Phi$ be the set of positive roots corresponding to our choice of Borel subgroup $\mbf B$. The standard parabolics $\mbf P \supset \mbf B$ correspond bijectively with choices of subsets $I \subset \Delta$. The choice $I = \emptyset$ gives $\mbf P = \mbf B$ and the choice $I = \Delta$ gives $\mbf P = \mbf G$. We will also use the notation $\Phi_{/L} \subset X(\mbf T_{/L})$ for the set of roots of $\mbf T_{/L}$, $\Phi_{/L}^+$ for those positive with respect to $\mbf B_{/L}$, etc.

The dual $\mf t_L^\ast = \Hom_L(\mbf t_L, L) = \Hom_{\mbf Q_p}(\mf t, L)$ is the set of weights of $\mf t_L$. The roots $\Phi_{/L}$ of $\mbf T_{/L}$ will be naturally identified with the subset of $\mf t_L^\ast$ consisting of the roots of $\mf t_L$ acting on $\mf g_L$ in the sense that we will use the {\em same} letter (usually $\alpha$) for either.  If $\lambda \in \mf t_L^\ast$ is a weight and $\eta: L \rightarrow \mf t_L$ is a co-weight then we write $\langle \lambda , \eta \rangle$ for the canonical element of $L$ given by $(\lambda \circ \eta)(1)$. If $\chi: T \rightarrow L^\times$ is a locally analytic character we write $d\chi: \mf t_L \rightarrow L$ for its derivative.

\section{Reminders}

\subsection{Verma modules}
\label{subsec:verma-reminder}
Throughout this subsection, we let $\mf g$ be a split reductive Lie algebra over a field $L$ of characteristic zero.\footnote{All the results in this subsection will be applied to $\mf g\otimes_{\mbf Q_p} L$ for $\mf g$ a $\mbf Q_p$-Lie algebra as in Section \ref{sec:verma-mods-fd}.} We will recall the theory of Verma modules for $\mf g$. Write $\mf t$ for a split Cartan subalgebra and $\mf b \supset \mf t$ for a Borel subalgebra containing $\mf t$. A standard parabolic is a parabolic subalgebra $\mf p \subset \mf g$ such that $\mf p \supset \mf b$. We write $\mf t^{\ast}$ for the $L$-linear dual of $\mf t$. We let $\Phi \subset \mf t^{\ast}$ be the set of roots of $\mf t$. The choice of Borel $\mf b$ determines a set of positive roots $\Phi^+$. Fix a set of simple roots $\Delta \subset \Phi^+$. Then each standard parabolic $\mf p$ is determined by a subset $I \subset \Delta$ as in \cite[Section 9.1]{hum}. The choice $I = \emptyset$ gives $\mf p = \mf b$ and the choice $I = \Delta$ gives $\mf p = \mf g$. Write $\rho_0 \in \mf t^{\ast}$ for the half sum of the positive roots.

If $\mf p$ is a standard parabolic then we write $\mf p = \mf n_{\mf p} \oplus \mf l_{\mf p}$ for its decomposition into the maximal nilpotent subalgebra and the Levi factor. Following \cite{os}, we define the category $\mathcal O^{\mathfrak p}$ to be the full subcategory of left $\mrm U(\mf g)$-modules $M$ which satisfy the following properties:
\begin{enumerate}
\item$M$ is finite-type over $\mrm U(\mf g)$.
\item $M$ is the direct sum of weight spaces $M_{\lambda} := \{m \in M \mid X\cdot m = \lambda(X)m \text{ for all $X \in \mf t$}\}$ as $\lambda$ runs over elements of $\mf t^{\ast}$.
\item For each vector $m \in M$, the $L$-vector space $\mrm U(\mf p)\cdot m$ is finite-dimensional.
\end{enumerate}
This an analog of the category $\mc O^{\mf p}$ defined in \cite[Chapter 9]{hum} which is suited to the possibility that $L$ itself may not be algebraically closed. For example, the second and third axioms together imply that $M$ is also a direct sum of finite-dimensional $\mrm U(\mf l_{\mf p})$-modules which are themselves absolutely irreducible. As a consequence, the results in \cite{hum} are valid to cite in our setting. Note that if $\mf p \supset \mf q$ are two choices of parabolic subgroups then $\mc O^{\mf p}$ is a full subcategory of $\mc O^{\mf q}$. In particular, $\mc O^{\mf b} \supset \mc O^{\mf p}$ for all standard parabolics $\mf p$ ($\mc O^{\mf b}$ is usually written plainly as $\mc O$). The elements in $\mc O^{\mf g}$ are the finite-dimensional representations of $\mathfrak g$ on which $\mathfrak t$ acts diagonalizably.

\begin{exam}
Let $W$ be an $L$-linear finite dimensional representation of $\mf l_{\mf p}$ on which $\mf t$ acts diagonalizably. Consider $W$ as a left module over $\mrm U(\mf l_{\mf p})$. By inflation it defines a left $\mathrm U(\frak p)$-module. We define the {\em generalized Verma module}
\begin{equation*}
M_{\mf p}(W) := \mrm U(\mf g) \otimes_{\mrm U(\mf p)} W.
\end{equation*}
The $\mrm U(\mf g)$-module $M_{\mf p}(W)$ lies in the full subcategory $\mc O^{\mf p}$. Note that we could also use the opposite parabolic $\mf p^-$, in which case we get $M_{\mf p^{-}}(W)$. 
\end{exam}

If $\alpha \in \Phi$ is a root of $\mf t$ then write $\alpha^\vee$ for its co-root. If $\mu \in \mf t^{\ast}$ is a weight then write $\langle \mu, \alpha^\vee \rangle := (\mu\circ\alpha^\vee)(1) \in L$. If $\mf p$ corresponds to the choice of simple roots $I \subset \Delta$ then write
\begin{equation*}
\Lambda_{\mf p}^+ = \{ \lambda \in \mf t^{\ast} \mid \langle \lambda + \rho_{0} , \alpha^\vee \rangle \in \mbf Z_{> 0} \text{ for all $\alpha \in I$} \}.
\end{equation*}
The elements $\lambda \in \Lambda_{\mf p}^+$ parameterize finite-dimensional absolutely irreducible $\mrm U(\mf l_{\mf p})$-modules on which $\mathfrak t$ acts diagonalizably \cite[Section 9.2]{hum}. We denote this correspondence by $\lambda \leftrightarrow W_\lambda$. In the case where $\mf p = \mf b$, every $\lambda$ is in $\Lambda_{\mf b}^+$ and  $W_\lambda$ is just the weight $\lambda$ itself. The salient facts are given by the following.
\begin{prop}\label{prop:verma-facts}
Suppose that $\lambda \in \mf t^{\ast}$.
\begin{enumerate}
\item The Verma module $M_{\mf b}(\lambda)$ has finite length in $\mathcal O^{\mathfrak b}$ and a unique irreducible quotient $L(\lambda)$.
\item Every irreducible object in $\mc O^{\mf p}$ is of the form $L(\lambda)$ for some $\lambda \in \Lambda_{\mf p}^+$.
\item If $\lambda \in \Lambda_{\mf p}^+$ then there is a unique $\mrm U(\mf g)$-module quotient $M_{\mf b}(\lambda) \twoheadrightarrow M_{\mf p}(W_\lambda)$ and $L(\lambda)$ is the unique irreducible quotient of $M_{\mf p}(W_\lambda)$.
\item If $\lambda \in \Lambda_{\mf p}^+$ then the irreducible constituent $L(\lambda)$ of $M_{\mf p}(W_\lambda)$ appears exactly once as a subquotient.
\item If $\lambda \in \Lambda_{\mf p}^+$ and $e_\lambda^+ \in M_{\mf b}(\lambda)$ is the highest weight vector then the $\mrm U(\mf p)$-submodule generated by the image of $e_{\lambda}^+$ in the quotient map $M_{\mf b}(\lambda) \twoheadrightarrow L(\lambda)$ is canonically isomorphic to $W_{\lambda}$.
\end{enumerate}
\end{prop}
\begin{proof}
This is all well-known so we mostly provide references. Point (1) follows from $M_{\mf b}(\lambda)$ being a highest weight module (see \cite[Corollary 1.2]{hum} and \cite[Theorem 1.2(f)]{hum}). The point (2) follows from \cite[Theorem 1.3]{hum} and \cite[Proposition 9.3]{hum}. We pause on (3) for the moment. The point (4), reduces to the case $\mf b = \mf p$ by (3). In that case, it is a consequence of highest weight theory.

Let us prove (3) and (5) now. The representation $W_\lambda$ has a highest weight vector $v_\lambda^+$. The map $e_{\lambda}^+ \mapsto v_\lambda^+$ extends to a $\mrm U(\mf g)$-equivariant map
\begin{equation*}
M_{\mf b}(\lambda) = \mrm U(\mf g) \otimes_{\mrm U(\mf b)} e_\lambda^+  \twoheadrightarrow \mrm U(\mf g) \otimes_{\mrm U(\mf p)} W_\lambda =M_{\mf p}(W_\lambda)
\end{equation*}
(see \cite[Section 9.4]{hum}) and we naturally have a commuting diagram
\begin{equation*}
\xymatrix{
M_{\mf b}(\lambda) \ar[r] \ar[d] & L(\lambda)\\
M_{\mf p}(W_{\lambda}) \ar[ur]
}
\end{equation*}
whose arrows are all surjective. This proves (3). To prove (5), let $W$ be the $\mrm U(\mf p)$-submodule generated by the image of $e_{\lambda}^+$ inside $L(\lambda)$. Then, consider $W_\lambda \subset M_{\mf p}(W_\lambda)$. This is a $\mrm U(\mf p)$-stable subspace. Since $v_\lambda^+ \in W_{\lambda}$, the image of $W_\lambda$ in $L(\lambda)$ is non-zero and contains the image of $e_\lambda^+$. Since $W_{\lambda}$ is an irreducible $\mrm U(\mf p)$-module we deduce that $W_{\lambda} \hookrightarrow L(\lambda)$ is a $\mrm U(\mf p)$-submodule that contains $W$. Again, since $W_{\lambda}$ is irreducible we deduce that $W = W_{\lambda}$.
\end{proof}

If $M$ is an $L$-vector space then $M^{\ast}$ denotes its $L$-linear dual. If $M$ has an action of $\mrm U(\mf g)$ then there are two separate actions we may endow $M^{\ast}$ with. To uniformly explain this, suppose that $s: \mrm U(\mf g) \rightarrow \mrm U(\mf g)$ is an $L$-linear anti-automorphism. In that case we can define a {\em left} action of $\mrm U(\mf g)$ on $M^{\ast}$ by
\begin{equation*}
(X\cdot_s f)(m) := f(s(X) m)
\end{equation*}
for $X \in \mrm U(\mf g)$, $f \in M^{\ast}$ and $m \in M$. The two choices for $s$ we will use are:\footnote{We use the $d$-notation to indicate our intention in Section \ref{sec:verma-mods-fd} to apply this to the derivatives of inversion and transposition on a split connected reductive algebraic group.}
\begin{itemize}
\item The anti-automorphism $d\iota$ on $\mrm U(\mf g)$ which acts by $-1$ on $\mf g$ (see \cite[\S 3.1]{br} for example), or
\item the anti-automorphism $d\tau$ on $\mrm U(\mf g)$ which fixes $\mf t$ point-by-point and defines an isomorphism $d\tau: \mf g_\alpha \simeq \mf g_{-\alpha}$ for every non-zero root $\alpha$ (see \cite[\S 0.5]{hum}).
\end{itemize}

We write $M^{\ast,d\iota}$ for the $\mrm U(\mf g)$-module $M^{\ast}$ equipped with the $d\iota$-action described above, and we write $M^{\ast,d\tau}$ when we equip $M^{\ast}$ with the action of $\mrm U(\mf g)$ by $d\tau$.

Now suppose that in addition to a left $\mrm U(\mf g)$-module structure, $M$ is also a weight module for the action of $\frak t$. For each weight $\lambda \in \mf t^{\ast}$, $M_\lambda$ is a direct summand of $M$. So, we consider the surjection $M \twoheadrightarrow M_\lambda$ and the corresponding $L$-linear inclusion $j_\lambda: (M_\lambda)^{\ast} \hookrightarrow M^{\ast}$. It is elementary to check that for each $\lambda \in \mf t^{\ast}$,
\begin{equation*}
(M^{\ast,\tau})_\lambda = j_{\lambda}((M_\lambda)^{\ast}) = (M^{\ast,\iota})_{-\lambda}.
\end{equation*}
If we define, just as an $L$-vector space, $M^{\ast,\infty} \subset M^{\ast}$ by
\begin{equation*}
M^{\ast,\infty} = \bigoplus_{\lambda \in \mf t^{\ast}_L} j_\lambda((M_\lambda)^{\ast})
\end{equation*}
then we have natural identifications of vector spaces (\textit{not} $\mrm U(\mf g)$-modules)
\begin{equation}\label{eqn:n-power-kills}
\left(M^{\ast,\iota}\right)[{\mf n_{\mf b^-}^\infty}]  = M^{\ast,\infty} = \left(M^{\ast,\tau}\right)[{\mf n_\mf b^{\infty}}].
\end{equation}
Here, if $V$ is an $L$-vector space with an $L$-linear action of a Lie algebra $\mf h$ and $k\geq 0$ is an integer then write $V[{\mf h^k}]$ for the subspace of $v \in V$ annihilated by every monomial $X_1\dotsb X_k \in \mrm U(\mf h)$ with $X_i \in \mf h$ (not necessarily distinct). We write $V[{\mf h^{\infty}}]$ for the union of the $V[{\mf h^k}]$ as $k\geq 0$.  If $M$ is a weight module then $(M^{\ast,\tau})[{\mf n_{\mf b}^k}]$ is a $\mrm U(\mf b)$-stable subspace and their union $(M^{\ast,\tau})[{\mf n_{\mf b}^\infty}]$ is $\mrm U(\mf g)$-stable. Similarly for the first space in \eqref{eqn:n-power-kills}.
\begin{defi}\label{defi:duality}
Suppose that $M \in \mc O^{\mf b}$. 
\begin{enumerate}
\item The internal dual, denoted $M^\vee$, is the left $\mrm U(\mf g)$-module $M^\vee := (M^{\ast,\tau})[{\mf n_{\mf b}^\infty}]$. 
\item The opposite dual, denoted $M^-$, is the left $\mrm U(\mf g)$-module $M^- := \left(M^{\ast,\iota}\right)[{\mf n_{\mf b^-}^\infty}]$.
\end{enumerate}
\end{defi}

It must be checked what category the separate duals lie in.  Note that since $\mc O^{\mf b} \supset \mc O^{\mf p}$, the dualities are defined on, but {\em a priori}  may not preserve, the category $\mc O^{\mf p}$.

\begin{lemm}\label{lemm:dualities}
Suppose that $M \in \mc O^{\mf p}$.
\begin{enumerate}
\item The internal dual $M^\vee$ lies in $\mc O^{\mf p}$.
\item The opposite dual $M^-$ lies in $\mc O^{\mf p^-}$.
\item Both dualities are exact, contravariant, equivalences of categories.
\item We have a commutation relation $(M^\vee)^- \simeq (M^-)^\vee$ in $\mc O^{\mf p^-}$.
\end{enumerate}
\end{lemm}
\begin{proof}
The underlying vector space of either duality is given by $M^{\ast,\infty} \subset M^{\ast}$. Since $\mf n_{\mf p} \subset \mf n_{\mf b}$ (and the same with the opposites substituted), the first two points are clear by definition. Furthermore, it also shows that $(M^\vee)^\vee \simeq M$ and $(M^-)^- \simeq M$, making (3) clear. To prove (4), the underlying vector space of $(M^{\vee})^-$ is $\bigoplus_{\lambda \in \mathfrak t_{L}^\ast} ((M^{\ast,\tau})^{\ast,\iota})_{-\lambda}$ and for $(M^-)^{\vee}$ the role of $\iota$ and $\tau$ is switched. Thus there is an obvious identification of either as linear subspaces of the double dual $M^{\ast\ast}$, and it is also clear that the actions of $d\tau$ and $d\iota$ are compatible.
\end{proof}

\subsection{Locally analytic representations}
\label{subsec:locally-anal-reminder}

We always equip a finite-dimensional $\mbf Q_p$-vector space with its natural topology, the unique finest convex one. Thus linear maps between such spaces are automatically continuous and any linear subspace is closed.

We refer to \cite[Section 3]{st-locdist} as a reference in what follows. Let $H$ be $p$-adic Lie group, i.e.\ a locally $\mbf Q_p$-analytic group, with Lie algebra $\mf h$. We shorten the term locally $\mbf Q_p$-analytic representation of $H$ to just locally analytic representation of $H$.  

There is a partially defined exponential map $\exp: \mf h \dashrightarrow H$ converging on an open neighborhood of $0$ in $\mf h$.  If $V$ is a locally analytic representation of $H$ then there is an induced continuous and $\mbf Q_p$-linear action of $\mf h$ on $V$ given by
\begin{equation*}
X(v) := {d \over dt}\bigg|_{t=0} \exp(tX)(v) \;\;\;\;\; (X \in \mf h, v \in V).
\end{equation*}
This makes $V$ a left module over $\mrm U(\mf h)$. If $V$ has the extra structure of a vector space over $L$ where $L/\mbf Q_p$ is a finite extension then clearly this gives $V$ the structure of a left module over $\mrm U(\mf h\otimes_{\mbf Q_p} L)$.

For each $v \in V$ there exists an open neighborhood $\Omega_v \subset \mf h$ such that the following hold:
\begin{enumerate}
\item $\exp$ is defined on $\Omega_v$,
\item if $X \in \Omega_v$ then $\sum {1\over n!} X^n(v)$ converges, and
\item if $X \in \Omega_v$ then $\exp(X)(v) = \sum {1\over n!} X^n(v)$.
\end{enumerate}
Write $\Ad: H \ra \GL(\mf h)$ for the adjoint representation. If $V$ is a locally analytic representation of $H$ then $h(X(v)) = (\Ad(h) X)(h(v))$ for all $h \in H$, $X \in \mf h$ and $v \in V$.
\begin{lemm}\label{lemm:convergence}
Let $V$ be a locally analytic representation of $H$. Suppose that $X \in \mf h$, $v \in V$ and $\sum {1\over n!}X^n(v)$ converges. Then, for every $h \in H$ the sum $\sum {1\over n!}(\Ad(h)X)^n(hv)$ converges and
\begin{equation*}
h\cdot\left(\sum {1\over n!} X^n(v)\right) = \sum {1\over n!}(\Ad(h)X)^n(hv).
\end{equation*}
\end{lemm}
\begin{proof}
This is clear by approximating by finite sums, and the fact mentioned above that $h(X(v)) = (\Ad(h)X)(hv)$ for all $h \in H$, $X \in \mf h$ and $v \in V$.
\end{proof}
We now turn to the larger context of this paper. So, let $\mbf G$ be a split connected reductive group over $K$ with maximal torus $\mbf T$ and Borel subgroup $\mbf B$. Recall that $\mbf N$ is the unipotent radical of the Borel subgroup $\mbf B$. Since $\mf n$ is nilpotent, the exponential map actually defines a bijection of algebraic varieties $\exp:\underline{\mf n} \simeq \mbf N$ (here $\underline{\mf n}$ means the underlying affine variety of the vector space $\mf n$) and the group law on $\mbf N$ is induced from the formula of Baker--Campbell--Hausdorff (compare with \cite[Section 2.5]{em3}). 

We use Roman letters $G$, $T$, $B$, etc. for the $K$-points of algebraic groups $\mbf G$, $\mbf T$, $\mbf B$, etc. over $K$, and we view them as locally $\mbf Q_p$-analytic groups by restriction of scalars on the underlying algebraic groups. The respective Lie algebras are insensitive to the restriction of scalars. In particular, the $\mbf Q_p$-analytic exponential map $\exp:\mf n \simeq N$ converges everywhere by the previous paragraph.

\begin{lemm}\label{lemm:open-nbhd}
If $\Omega \subset \mf n$ is a non-empty open neighborhood of $0$ then the orbit of $\Omega$ under the adjoint action of $T$ is all of $\mf n$.
\end{lemm}
\begin{proof}
We must show that if $Y \in \mf n$ then there exists a $t \in T$ such that $\Ad(t)(Y) \in \Omega$. For each root $\alpha$ of $\mbf T$ which is positive with respect to $\mbf B$, write $X_{\alpha} \in \mf n$ for the corresponding root vector. Then we can write $Y = \sum c_{\alpha} X_{\alpha}$, with $c_{\alpha} \in K$ and the sum running over positive roots. Thus, since $\Ad(t)X_{\alpha} = \alpha(t)X_{\alpha}$ for each $\alpha$, it suffices to show that for each $\varepsilon > 0$ there exists a $t \in T$ such that $|\alpha(t)| < \varepsilon$ for all positive simple roots $\alpha$. Since $\alpha$ is multiplicative, it suffices to check $\varepsilon = 1$ and the existence of such a $t$ is well-known.\footnote{For example, if $\Delta$ is the set of simple roots of $\mbf T$ corresponding to $\mbf B$ then extend $\Delta$ to a basis $\Delta'$ of $X(\mbf T)\otimes_{\mbf Z} \mbf Q$. If $\alpha \in \Delta$, let $\omega_{\alpha} \in Y(\mathbf T)\otimes_{\mbf Z} \mbf Q$ be the dual vector (with respect to $\Delta'$). If $\omega_0 := \sum_{\alpha \in \Delta} \omega_{\alpha}$ then for $n \gg 0$, $\omega = n\omega_0 \in Y(\mbf T)$. If $\varpi_{K}$ is a uniformizer in $K$ then $t = \omega(\varpi_K)$ works because $\alpha(\omega(\varpi_K)) = \varpi_K^n$.}
\end{proof}

\begin{prop}\label{prop:exponentiate-action}
If $V$ is a finite-dimensional locally analytic representation of $B$ then for all $Y \in \mf n$ and all $v \in V$ the sum $\sum {1\over n!}Y^n(v)$ converges and
\begin{equation}\label{eqn:exponential-equal}
\exp(Y)(v) = \sum {1\over n!}Y^n(v).
\end{equation}
\end{prop}
\begin{proof}
Since $V$ is finite-dimensional and $\mathfrak b$ is solvable, we can upper triangularize the action of $\mathfrak b$ after extending scalars (Lie's theorem). But, since $\mathfrak n$ is diagonalized under the bracket action of $\mathfrak t$ it is clear the $\mathfrak n$ acts nilpotently on $V$ after, and thus also before, extending scalars. Thus the sum $\sum {1\over n!}Y^n(v)$ converges for all $v\in V$ and $Y\in \mathfrak n$.

If $v_1,\dotsc,v_r$ is a basis of $V$ then there exists non-empty open neighborhoods $\Omega_i \subset \mathfrak n$ of 0 on which \eqref{eqn:exponential-equal} holds for all $Y \in \Omega_i$ and $v=v_i$. By linearity, \eqref{eqn:exponential-equal} holds for all $v \in V$ and $Y \in \Omega:=\bigcap \Omega_i$. But then Lemma \ref{lemm:convergence} implies \eqref{eqn:exponential-equal} holds for all $v \in V$ and $Y$ in the orbit of $\Omega$ under the adjoint action of $T$, which is all of $\mathfrak n$ by Lemma \ref{lemm:open-nbhd}.
\end{proof}

\begin{rema}
Proposition \ref{prop:exponentiate-action} is an analog of the well-known statement:\ if $V$ is a smooth admissible representation of $B$ then $N$ acts trivially on $V$ (see \cite[Lemma 13.2.3]{boy} for example). 
\end{rema}

\begin{coro}\label{coro:finite-dimensional-equivariance}
Let $V$ and $W$ be finite-dimensional locally analytic representations of $B$.
\begin{enumerate}
\item The natural inclusion $\Hom_B(V,W) \subset \Hom_{(\mf b,T)}(V,W)$ is an equality.
\item if $V_0 \subset V$ is a $(\mf b,T)$-stable subspace then $V_0$ is a $B$-stable subspace.
\end{enumerate}
\end{coro}
\begin{proof}
Since $B = NT$ it suffices in (1) to show that every $(\mf b,T)$-equivariant linear map $V \ra W$ is $N$-equivariant, and in (2) it suffices to show that $V_0$ is $N$-stable. Both are immediate from Proposition \ref{prop:exponentiate-action} after remarking that linear maps are continuous in (1) and the subspace $V_0$ in (2) is closed (since everything is finite-dimensional).
\end{proof}

We make a mild generalization of Proposition \ref{prop:exponentiate-action} and Corollary \ref{coro:finite-dimensional-equivariance} suitable for the Verma modules we construct in the next section. Recall that if $(M_i)_{i \in I}$ is a directed system of locally analytic representations of a $p$-adic Lie group $H$ with injective transition maps then when we equip $M = \varinjlim M_i$ with its locally convex inductive limit topology and its natural action of $H$, it becomes an $L$-linear locally analytic representation itself (see \cite[Lemma 3.1.4(iii)]{em2} for example). If each $M_i$ is finite-dimensional then this is the same as equipping $M$ with its finest convex topology.
\begin{prop}\label{prop:directed-system}
Suppose that $M$ is a locally analytic representation of $B$ which is isomorphic to the direct limit $M = \varinjlim M_i$ of an inductive system $(M_i)_{i \in I}$ of finite-dimensional locally analytic representations $M_i$ of $B$ with injective transition maps.
\begin{enumerate}
\item If $m \in M$ and $Y \in \mf n$ then $\sum {1\over n!} Y^n(m)$ converges and is equal to $\exp(Y)(m)$.
\item If $M' = \lim M_i'$ is another such directed system then the natural inclusion
\begin{equation*}
\Hom_B(M,M') \subset \Hom_{(\mf b,T)}(M,M')
\end{equation*}
is an equality.
\item A linear subspace $M_0 \subset M$ is $B$-stable if and only if it is $(\mf b,T)$-stable.
\end{enumerate}
\end{prop}
\begin{proof}
Write $\psi_i: M_i \ra M$ for the canonical maps arising from the directed system. They are continuous, $B$-equivariant and every $m\in M$ is in the image of some $\psi_i$.

Let us prove (1). Let $m \in M$ and choose an $i$ such that $m = \psi_i(m_i)$ for some $m_i \in M_i$. The $B$-representation $M_i$ satisfies the hypotheses of Proposition \ref{prop:exponentiate-action}. Since $\psi_i$ is continuous we conclude that $\sum {1\over n!} Y^n(m) = \psi_i\left(\sum {1\over n!} Y^n(m_i)\right)$ converges in $M$, and since $\psi_i$ is $B$-equivariant we conclude it is equal to $\exp(Y)(m) = \psi_i(\exp(Y)(m_i))$.

To prove (2) we let $f: M \ra M'$ be a $(\mf b,T)$-equivariant map and we need to show that $f$ is $B$-equivariant. Since $B = NT$ it suffices to show that $f$ is $N$-equivariant. Let $m \in M$ and suppose that $g \in N$. Choose $Y \in \mf n$ such that $g = \exp(Y)$. By part (1) of this proposition we have $g(m) = \exp(Y)(m) = \sum {1\over n!} Y^n(m)$. Note that $f$ is automatically continuous since $M$ and $M'$ are each equipped with the directed limit topology of finite-dimensional subspaces. In particular, since $f$ is assumed to be $\mrm U(\mf n)$-equivariant we deduce that $f(g(m)) = \sum {1\over n!} Y^n(f(m))$. On the other hand, part (1) of this proposition applied to $M'$ implies this is equal to $g(f(m))$. This concludes the proof of (2).

For (3), any subspace $M_0 \subset M$ is automatically closed because $M$ is equipped with its finest convex topology. Thus, if $M_0$ is $B$-stable then it is clearly $(\mf b,T)$-stable also. Suppose we only know that $M_0$ is $(\mf b,T)$-stable. For each $i$, the subspace $\psi_i^{-1}(M_0) \subset M_i$ is a $(\mf b,T)$-stable subspace of $M_i$, and thus Corollary \ref{coro:finite-dimensional-equivariance}(2) implies that $\psi_i^{-1}(M_0)$ is $B$-stable. Now, if $m_0 \in M_0$ we can choose $i$ such that $m_0 = \psi_i(m_{i})$ for some $m_i \in M_i$, necessarily in $\psi_i^{-1}(M_0)$. In particular, for each $b \in B$, $bm_i \in \psi_i^{-1}(M_0)$. But $\psi_i$ is $B$-equivariant and thus $b(m) = \psi_i(bm_i) \in \psi_i(\psi_i^{-1}(M_0)) \subset M_0$. This concludes the proof of (3).
\end{proof}

\section{Composition series of some Verma modules}\label{sec:verma-mods-fd}

We maintain the notations used in Section \ref{subsec:locally-anal-reminder}, and we will begin studying certain locally analytic actions of parabolic subgroups $P$ on $L$-vector spaces. Thus we will often refer to Section \ref{subsec:verma-reminder} with regard to the Lie algebras $\mf g_L$, $\mf t_L$, etc.\ We denote by $\Ad$ throughout the $L$-linear representation of $G$ on $\mf g_L$ induced by the usual adjoint action of $G$ on $\mf g$, viewed only $\mbf Q_p$-linearly and then extended $L$-linearly. Under the decomposition $\mf g_L \simeq \bigoplus_{\sigma \in \mc S} \mf g\otimes_{K} L_\sigma$, on the $\sigma$-component of the sum we have $G$ acting $K$-linearly on $\mf g$ and the action is extended $L$-linearly via $\sigma: K \rightarrow L$. We also use $\Ad$ for the action of $G$ on $\mrm U(\mf g_L)$.

\subsection{An extension of Verma module theory}\label{subsec:verma-fd}
\begin{defi}\label{defi:loc-anal-defi}
An $L$-linear locally analytic $(\mf g,P)$-module is an $L$-linear locally analytic representation $M$ of $P$ equipped with an auxiliary continuous action of $\mrm U(\mf g_L)$ such that
\begin{enumerate}
\item The derived action of $P$ on $M$ agrees with the restriction of the action of $\mrm U(\mf g_L)$ to the action of $\mrm U(\mf p_L)$.
\item For all $X \in \mrm U(\mf g_L)$, $g \in P$ and $m \in M$ we have $(\Ad(p)X)(p(m)) = p(X(m))$.
\end{enumerate}
A morphism of locally analytic $(\mf g,P)$-modules is a continuous $L$-linear morphism equivariant for the actions of $P$ and $\mrm U(\mf g_L)$.
\end{defi}
For example if $M$ is an $L$-linear locally analytic representation of $G$ itself then the restriction of the action of $G$ to the action of $P \subset G$ results in an $L$-linear locally analytic $(\mf g,P)$-module. We remark that if $M$ is a locally analytic $(\mf g,P)$-module and $k\geq 0$ then $M[\mf n_{P,L}^k] \subset M$ is a $P$-stable subspace of $M$ and $\mf n_{P^-,L}^k\cdot M$ is $L_P$-stable (but may be not $P$-stable).

Recently, Orlik and Strauch have studied interesting examples of locally analytic $(\mf g,P)$-modules \cite{os-old,os}.
\begin{defi}\label{defi:OP}
The category $\mc O^P$ is the full subcategory of $L$-linear locally analytic $(\mf g,P)$-modules $M$ such that:
\begin{enumerate}
\item $M$ is topologically the union of finite-dimensional $P$-stable subrepresentations.
\item The $\mrm U(\mf g_L)$-module $M$ lies in the BGG category $\mc O^{\mf p_L}$ defined in Section \ref{subsec:verma-reminder}.
\end{enumerate}
\end{defi}
\begin{rema}
By Definition \ref{defi:OP}(1), each $M$ in $\mc O^P$ is equipped with its finest convex topology. Thus continuity is automatic for morphisms in $\mc O^P$.
\end{rema}
\begin{rema}
Objects in $\mathcal O^G$ are finite-dimenisonal $L$-linear locally analytic representations of $G$ on which $\mathfrak t_L$ acts diagonalizably.
\end{rema}
\begin{exam}\label{exam:generalized-verma-modules}
Suppose that $U$ is an $L$-linear finite-dimensional locally analytic representation of the group $P$ and the Lie algebra $\mf t_L$ acts diagonalizably on $U$. Then we endow $M_{\mf p_L}(U) := \mrm U(\mf g_L) \otimes_{\mrm U(\mf p_L)} U$ with an action of $P$ by 
\begin{equation*}
g(X\otimes u) = \Ad(g)X\otimes g(u), \;\;\;\;\;\;\; g \in P,\;\; X \in \mrm U(\mf g_L), \;\;u \in U.
\end{equation*}
One must check that this is well-defined (it respects the relations implicit in a tensor product) but that is easy. We equip $M_{\mf p_L}(U)$ with its finest convex topology. This is reasonable:\ multiplication induces an isomorphism $\mrm U(\mf p_L)\otimes_L \mrm U(\mf n_{P^-,L}) \simeq \mrm U(\mf g_L)$, and so $M_{\mf p_L}(U) \simeq \mrm U(\mf n_{P^-,L})\otimes_L U$ (algebraically) and a homeomorphism when $\mrm U(\mf n_{P^-,L})$ is equipped its finest convex topology (which is natural) and the tensor product is given its inductive tensor product topology. We have made $M_{\mf p_L}(U)$ into an $L$-linear locally analytic $(\mf g,P)$-module (since the adjoint action is $\mbf Q_p$-algebraic and the action on $U$ is locally analytic).

One can see $M_{\mf p_L}(U)$ lies in $\mc O^P$. Indeed, it is clear that the underlying $\mrm U(\mf g_L)$-module $M_{\mf p_L}(U)$ lies in $\mc O^{\mf p_L}$ (Definition \ref{defi:OP}(2)). In particular, $M_{\mf p_L}(U) = M_{\mf p_L}(U)[\mf n_{P,L}^\infty]$. Since we endow $M_{\mf p_L}(U)$ with its finest convex topology, it is thus topologically the union of the finite-dimensional $P$-stable subspaces $M_{\mf p}(U)[\mf n_{P,L}^k]$ ranging over $k\geq 0$. Thus condition (1) in Definition \ref{defi:OP} is satisfied also. We refer to $M_{\mf p}(U)$ as the generalized Verma module associated to $U$.
\end{exam}
\begin{rema}
Every object $M$ in $\mc O^P$ is a quotient of $M_{\mf p_L}(U)$ for some $U \subset M$ a $P$-stable subrepresentation (see \cite[Lemma 2.8]{os}).
\end{rema}
If $M,M' \in \mc O^P$ then we write $\Hom_{\ast}(M,M')$ to mean $L$-linear morphisms in $\mc O^P$ equivariant for the action(s) of $\ast$. For example, $\Hom_{(\mf g,P)}(M,M')$ is just the set of morphisms in $\mc O^P$ (of course, being $L$-linear and $\mf g$-equivariant is the same as being $\mf g_L$-equivariant).
\begin{prop}\label{prop:homs-ok}
Let $M \in \mc O^P$.
\begin{enumerate}
\item If $M \in \mc O^P$ then for every $m \in M$ and $Y \in \mf n_P$ the sum $ \sum {1\over n!} Y^n(m)$ converges and $\exp(Y)(m) = \sum {1\over n!} Y^n(m)$.
\item If $M' \in \mc O^P$ also then the natural inclusions
\begin{equation*}
\Hom_{(\mf g,P)}(M,M') \subset \Hom_{(\mf g,L_P)}(M,M') \subset \Hom_{(\mf g,T)}(M,M')
\end{equation*}
are equalities.
\item If $M_0 \subset M$ is a linear subspace then $M_0$ is $(\mf p,T)$-stable if and only if $M_0$ is $P$-stable. In particular, $M_0$ is $(\mf g,T)$-stable if and only if $M_0$ is $(\mf g,P)$-stable.
\end{enumerate}
\end{prop}
\begin{proof}
For (1), the convergence statement is immediate because $M$ is in $\mathcal O^P$. The agreement with the exponential action is immediate from Proposition \ref{prop:directed-system}(1).

We will also use Proposition \ref{prop:directed-system} to show (2). We will separately prove that each inclusion is an equality. Suppose first that $f: M \ra M'$ is $L$-linear and $(\mf g,L_P)$-equivariant, and we want to show that $f$ is $P$-equivariant. Since $P = N_PL_P$ and $f$ is $L_P$-equivariant, it is enough to show $f$ is $N_P$-equivariant. Since $M$ and $M'$ are in $\mc O^P$, they are topologically an inductive limit of $P$-stable, and thus $B$-stable, subrepresentations. Since $\mf b \subset \mf g$ and $T \subset L_P$, $f$ is automatically $(\mf b,T)$-stable. By Proposition \ref{prop:directed-system}(2) we deduce that $f$ is $B$-equivariant. But this is enough to show $f$ is $N_P$-equivariant since $N_P \subset N \subset B$. This proves our first equality.

The proof of the second equality follows the same lines. Suppose that $f: M \ra M'$ is $(\mf g,T)$-equivariant. Let $\mbf B_{\mbf L_{\mbf P}} \subset \mbf L_{\mbf P}$ be the Borel subgroup $\mbf L_{\mbf P} \cap \mbf B$ of $\mbf L_{\mbf P}$ and let $\mbf B_{\mbf L_{\mbf P}}^{-}$ be its opposite. Then the group $L_P$ is generated as a group by $B_{L_P}$ and $B_{L_P}^-$.\footnote{This follows from the Bruhat decomposition \cite[Section 1.9]{jantzen} and the fact that Weyl group representatives can be taken from the group generated by $B_{L_P}$ and $B_{L_P}^-$ \cite[Sections 1.3-1.4]{jantzen}.} By Proposition \ref{prop:directed-system}(2), $f$ is $B_{L_P}$ and $B_{L_P}^-$-equivariant and thus $L_P$-equivariant by the previous sentence.

The proof of (3) follows the lines of (2). Indeed, if $M_0$ is $(\mf p,T)$-stable then $M_0$ is $(\mf p,L_P)$-stable using the fact that $L_P$ is generated by $B_{L_P}$ and $B_{L_P}^-$ and Proposition \ref{prop:directed-system}(3) applied to either Borel subgroup. Since $M_0$ is $(\mf p, L_P)$-stable, and it is automatically closed in $M$, it is $P$-stable by part (1) of this proposition. This concludes the proof.
\end{proof}

\begin{coro}\label{coro:fully-faithful}
If $\mbf P \subset \mbf Q$ are two standard parabolics then the forgetful inclusion $\mc O^Q \subset \mc O^P$ is fully faithful.
\end{coro}
\begin{proof}
If $\mbf P \subset \mbf Q$ and $M,M' \in \mc O^Q$ then by Proposition \ref{prop:homs-ok} we have
\begin{equation*}
\Hom_{(\mf g,Q)}(M,M') = \Hom_{(\mf g,T)}(M,M') = \Hom_{(\mf g,P)}(M,M'),
\end{equation*}
as desired.
\end{proof}

\subsection{Duality}\label{subsec:duality}
Suppose that $M$ is an $L$-linear locally analytic $(\mf g,P)$-module and let $M^{\ast}$ be the $L$-linear dual equipped with its strong dual topology. Similar to the duality discussion in Section \ref{subsec:verma-reminder}, we can consider two separate actions of $\mrm U(\mf g_L)$ and $L_P$ on $M^{\ast}$. Recall the the group $\mbf G$ comes equipped with an inversion map $\iota: \mbf G \rightarrow \mbf G$ and a ``transpose'' map $\tau: \mbf G \rightarrow \mbf G$ which depends on our choice of $\mathbf B \supset \mathbf T$ (\cite[Section II.1.15]{jantzen}). These induce $\mbf Q_p$-algebraic maps on $G$ which we write $g \mapsto g^{-1}$ and $g\mapsto {}^{\tau} g$. Then, for all $x \in L_P$, $X \in \mrm U(\mf g_L)$, $f \in M^{\ast}$ and $m \in M$ we declare
\begin{equation*}
(x\cdot_{\iota} f)(m) := f(x^{-1}(m)) \;\;\;\; \;\;\;\; (X\cdot_{\iota} f)(m) := f(d\iota(X)(m))
\end{equation*}
or
\begin{equation*}
(x\cdot_{\tau} f)(m) := f({}^{\tau}\! x(m)) \;\;\;\; \;\;\;\; (X\cdot_{\tau} f)(m) := f(d\tau(X)(m)),
\end{equation*}
where $d\iota$ and $d\tau$ denote the respective $\mbf Q_p$-linear differentials (this aligns with Section \ref{subsec:verma-reminder}).

We denote $M^{\ast}$ with the first action as $M^{\ast,\iota}$, and $M^{\ast}$ with the second action as $M^{\ast,\tau}$. It is immediately verified that these separate actions of $\mrm U(\mf g_L)$ and $L_P$ are compatible in the sense of Definition \ref{defi:loc-anal-defi}(2).\footnote{\label{footnote:trans-inv-relation}It is helpful to remember, or derive, that for all $g \in G$ and $X \in \mrm U(\mf g$) we have $\Ad(g)d\iota(X) = d\iota(\Ad(g)X)$ and $d\tau(\Ad(g)X) = \Ad({}^\tau\!g^{-1})d\tau(X)$}\label{footnote:commutation-relation} 

If $M$ is equipped with its finest convex topology (e.g. if $M$ is in $\mc O^P$) then the strong and weak topology on $M^{\ast}$ agree, and the action of $L_P$ on either $M^{\ast,\iota}$ or $M^{\ast,\tau}$ is continuous. If $M$ is itself finite-dimensional then it is also easily checked that the action of $L_P$ is locally analytic (since $\iota$ and $\tau$ are algebraic maps of the underlying algebraic group $\Res_{K/\mbf Q_p} \mbf L_P$) and the differential action of $\mrm U(\mf l_{P,L})$ agrees with the restriction action of $\mrm U(\mf g_L)$ to $\mrm U(\mf l_{P,L})$. In an effort to upgrade to an action of $P$ we make the following definition.

\begin{defi}\label{defi:duals}
Let $M \in \mc O^P$.
\begin{enumerate}
\item The internal dual is $M^{\vee} := (M^{\ast,\tau})[{\mf n_{P,L}^\infty}]$.
\item The opposite dual is $M^- := (M^{\ast,\iota})[{\mf n_{P^-,L}^\infty}]$.
\end{enumerate}
\end{defi}
Note that for $Y \in \mf n_P$ and $m \in M$ the infinite sum $\sum {1\over n!} Y^n(m)$ converges on $M^\vee$ because it is in fact a finite sum (similarly for $M^-$).
\begin{prop}
Let $M \in \mc O^P$.
\begin{enumerate}
\item $M^{\vee}$ and $M^-$ are $L$-linear locally analytic $(\mf g,L_P)$-modules. 
\item If we extend the action of $L_P$ on $M^\vee$ to $P = N_PL_P$ via $\exp(Y)(m) = \sum {1\over n!} Y^n(m)$ for $Y \in N_P$ and $m \in M$ then $M^\vee \in \mc O^P$. 
\item If we extend the action of $L_P$ on $M^-$ to $P^- = N_{P^-}L_P$ via $\exp(Y)(m) = \sum {1\over n!} Y^n(m)$ for $Y \in N_{P^-}$ and $m \in M$ then $M^{-} \in \mc O^{P^-}$.
\end{enumerate}
\end{prop}
\begin{proof}
The statements about $M^-$ are proved in the same way as for $M^\vee$, with only minor modifications, so we will discuss only $M^\vee$.

Let $k \geq 0$ be an integer. We first remark that $M^\vee$ is the union of $(M^{\ast,\tau})[{\mf n_P^k}]$ for all $k\geq 0$. Since $M$ is equipped with its finest convex topology, the strong and weak topology on $M^{\ast}$ agree. Thus, the induced topology on $M^\vee$ is the inductive limit topology on the union $\bigcup (M^{\ast,\tau})[{\mf n_P^k}]$. We claim that each element of the union is naturally an $L$-linear locally analytic representation of $L_P$. This would complete the proof of (1).

Since $M \in \mc O^P$ the subspace $\mf n_{P^-,L}^k \cdot M \subset M$ is $L_P$-stable and has finite co-dimension. In particular, the dual $\tau$-action on $(M/\mf n_{P^-,L}^kM)^{\ast,\tau}$ naturally defines an $L$-linear finite-dimensional locally analytic representation of $L_P$ (compatible with the action of $\mrm U(\mf l_{P,L})$). On the other hand, 
\begin{equation*}
(M^{\ast,\tau})[{\mf n_{P,L}^k}] = \{ f : M \ra L \mid \mf n_{P,L}^k \cdot f = 0 \} = \{ f : M \ra L \mid f(\mf n_{P^-,L}^k M) = 0 \}
\end{equation*}
(note the careful switch of the parabolic, due to the action of $\mrm U(\mf g_L)$ on $f$ being via $d\tau$) and thus we have a natural isomorphism
\begin{equation*}
(M/\mf n_{P^-,L}^k M)^{\ast,\tau} \simeq (M^{\ast,\tau})[{\mf n_{P,L}^k}].
\end{equation*}
of $L_P$-representations. This proves the claim.

For (2), we need to check that $M^\vee$ satisfies Definition \ref{defi:OP} once equipped with the auxiliary action of $N_P$. Definition \ref{defi:OP}(2) holds by Lemma \ref{lemm:dualities}. To see Definition \ref{defi:OP}(1), we observe that the exponentiated action of $\mf n_P$ preserves the subspaces $(M^{\ast,\tau})[{\mf n_{P,L}^k}] \subset M^\vee$. Thus $M^\vee$ is the topological union of $L$-linear finite-dimensional locally analytic $P$-stable subspaces.
\end{proof}

\begin{rema}
Exponentiating the action of $\mf n_P$ to obtain an action of $N_P$ in Definition \ref{defi:duals} may seem like a choice, but it is our only option in view of Proposition \ref{prop:homs-ok}(1).
\end{rema}

\begin{prop}\label{prop:duality-exact}
The duality functors $M\mapsto M^-$ and $M \mapsto M^\vee$ are exact contravariant functors and $(M^{\vee})^- \simeq (M^-)^{\vee}$ in $\mc O^{P^-}$.
\end{prop}
\begin{proof}
The associations are obviously functors. The exactness is a statement about the underlying vector spaces and so follows from Lemma \ref{lemm:dualities}(3). The commutation relation reduces, by Proposition \ref{prop:homs-ok}(2), to checking that $(M^\vee)^{-} \simeq (M^-)^{\vee}$ as $(\mf g,L_P)$-modules. But that is elementary given the definitions (compare with the proof of Lemma \ref{lemm:dualities}(4)).
\end{proof}

\subsection{An explicit realization of opposite duals for generalized Verma modules}
Assume throughout this section that $U$ is an $L$-linear finite-dimensional locally analytic representation of $L_P$. We may view $U$ by inflation either as a representation of $P$ or $P^-$. We are going to give a concrete interpretation of the opposite dual $M_{\mf p^-_L}(U^{\ast,\iota})^-$ of the generalized Verma module associated to the contragradient representation $U^{\ast,\iota}$ in terms of $U$-valued polynomial functions on $N_P$. 
\begin{rema}\label{rema:remove-diagonal-hyp}
Strictly speaking we only discussed $M_{\mf p}(U)$ (and its analogs) when $\mf t_L$ acts diagonalizably on $U$. However, the definition makes sense for general $U$ and since the following construction does not use the assumption on $\mf t_L$ we will omit that assumption.
\end{rema}

View $\mrm U(\mf g_L)$ and $U$ as left $\mrm U(\mf p^-_L)$-modules and let $\Hom_{\mrm U(\mf p^-_L)}(\mrm U(\mf g _L), U)^{\iota}$ be the space of left $\mrm U(\mf p^-_L)$-module morphisms $\mrm U(\mf g_L) \ra U$. The $\iota$ signifies that we view this as an $L$-vector space equipped with its usual action of $P^-$ via
\begin{equation*}
(g \cdot f)(X) := g\cdot f(\Ad(g^{-1}) X) \;\;\;\;\;\;\;\; (g \in P^-,\; X \in \mrm U(\mf g_L),\; f \in \Hom_{\mrm U(\mf p^-_L)}(\mrm U(\mf g_L), U)).
\end{equation*}
It is also a left $\mrm U(\mf g_L)$-module via right multiplication of functions:\ $(Y\cdot f)(X) = f(XY)$.

On the other hand, we can consider $U^{\ast,\iota}$ as an $L$-linear finite-dimensional locally analytic representation of $L_P$, inflate the action to $P^-$, and then form the generalized Verma module $\mrm U(\mf g_L) \otimes_{\mrm U(\mf p^-_L)} U^{\ast,\iota}$ as in Example \ref{exam:generalized-verma-modules}. Then, we get the $\iota$-action on the dual space 
\begin{equation*}
M_{\mf p^-_L}(U^{\ast,\iota})^{\ast,\iota} = \Hom_{L}(\mrm U(\mf g_L) \otimes_{\mrm U(\mf p^-_L)} U^{\ast,\iota}, L)^{\iota}
\end{equation*}
as in Section \ref{subsec:duality}.
\begin{lemm}\label{lemma:isomorphism-opposite-dual}
Let $U$ be an $L$-linear finite-dimensional locally analytic representation of $L_P$. There is a canonical isomorphism of $L$-vector spaces
\begin{align*}
\Hom_{\mrm U(\mf p^-_L)}(\mrm U(\mf g_L),U)^{\iota} &\simeq \Hom_{L}(\mrm U(\mf g_L)\otimes_{\mrm U(\mf p^-_L)} U^{\ast,\iota} , L)^{\iota}\\
f &\mapsto \Psi_f
\end{align*}
given by $\Psi_f(X\otimes \lambda) = \lambda(f(d\iota(X)))$ for all $X \in \mrm U(\mf g_L)$ and $\lambda \in U^{\ast}$. It is equivariant for the actions of $P^-$ and $\mrm U(\mf g_L)$ on either side.
\end{lemm}
\begin{proof}
We must check three things:\ 1) given $f$, $\Psi_f$ is well-defined on the tensor product, 2) $f\mapsto \Psi_f$ is equivariant for the actions of $P^-$ and $\mrm U(\mf g_L)$ and 3) that $f \mapsto \Psi_f$ is an isomorphism. The first two are formal manipulations given the definitions in the previous paragraph, so we leave them to the reader.

For the third point, we identify $U$ with its double dual $(U^{\ast,\iota})^{\ast,\iota}$ via $u \mapsto e_u$ where $e_u(\lambda) = \lambda(u)$ for all $\lambda \in U^{\ast}$ (this is $L_P$-equivariant). With this in mind it is easily checked that the map
\begin{align*}
\Hom_{L}(\mrm U(\mf g_L)\otimes_{\mrm U(\mf p^-_L)} U^{\ast,\iota} , L)^{\iota} &\ra \Hom_{\mrm U(\mf p^-_L)}(\mrm U(\mf g_L),U)^{\iota}\\
f &\mapsto \Phi_f,
\end{align*}
given by $\Phi_f(X)(\lambda) = f(d\iota(X) \otimes \lambda)$ for all $X \in \mrm U(\mf g_L)$ and $\lambda \in U^{\ast}$ is an inverse to $\Psi_f$.
\end{proof}

Recall from \cite[Section 2.5]{em3} that $C^{\pol}(N_P,L)$ is the space of $L$-valued polynomial functions on $N_P$. It is equipped with a natural action of $P$ where $N_P$ acts via the right regular action and $L_P$ acts via $(gf)(n) = f(g^{-1}ng)$. This action is locally analytic when we equip $C^{\pol}(N_P,L)$ with its finest convex topology. If $U$ is as above then we write $C^{\pol}(N_P,U) = C^{\pol}(N_P,L)\otimes_L U$. When we equip this tensor product with its inductive tensor product topology, the diagonal action of $L_P$ becomes locally analytic. Since $U$ is finite-dimensional, the natural topology on $C^{\pol}(N_P,U)$ is also the finest convex topology. According to \cite[Lemma 2.5.8]{em3} there is a canonical way of turning $C^{\pol}(N_P,U)$ into a locally analytic $(\mf g,P)$-module. In terms of this extension, we have:
\begin{prop}\label{prop:polynomia-opposite-dual}
If $U$ is an $L$-linear finite-dimensional locally analytic representation of $L_P$ then there is a canonical $(\mf g,P)$-equivariant isomorphism $C^{\pol}(N_P, U) \simeq M_{\mf p^-_L}(U^{\ast,\iota})^-$
\end{prop}
\begin{proof}
We begin by remarking that any $L$-linear $(\mf g,P)$-equivariant identification will be topological since both sides are equipped with their finest convex topologies. This being said, the isomorphism amounts to unwinding the extension of the $P$-representation on $C^{\pol}(N_P,U)$ to a $(\mf g,P)$-module as in \cite[Section 2.5]{em3}. Namely, \cite[(2.5.7)]{em3} naturally realizes $C^{\pol}(N_P,U)$, $P$-equivariantly, as the subspace
\begin{equation*}
C^{\pol}(N_P,U) \simeq \left[\Hom_{\mrm U(\mf p^-_L)}(\mrm U(\mf g_L), U)^{\iota}\right][{\mf n_{P,L}^\infty}] \subset \Hom_{\mrm U(\mf p^-_L)}(\mrm U(\mf g_L), U)^{\iota},
\end{equation*}
and this gives the induced action of $\mrm U(\mf g_L)$ on $C^{\pol}(N_P,U)$ (one must check that Emerton's actions agree with the $\iota$-action we are using, but that is formal). Then, Lemma \ref{lemma:isomorphism-opposite-dual} gives the identification
\begin{equation*}
C^{\pol}(N_P,U) \simeq \left[\Hom_{L}(\mrm U(\mf g_L)\otimes_{\mrm U(\mf p^-_L)} U^{\ast,\iota} , L)^{\iota}\right][\mf n_{P,L}^\infty] = M_{\mf p^-_L}(U^{\ast,\iota})^{-}.
\end{equation*}
This completes the proof.
\end{proof}

\subsection{Digression on infinitesimally simple representations}\label{subsec:digression}
This section uses the notations and conventions of the previous sections. For the reader, we note that we will eventually apply the results of this subsection to a Levi factor (which is still a split connected reductive group).
\begin{defi}
An $L$-linear finite-dimensional locally analytic representation $U$ of $G$ is called $\mf g$-simple if it is irreducible as a module over  $\mrm U(\mf g_L)$.
\end{defi}

If $U$ is an $L$-linear finite-dimensional locally analytic representation of $G$ then {\em a fortiori} it is an $L$-linear finite-dimensional locally analytic representation of either the Borel subgroup $B$ or its opposite $B^-$. In particular, the action of $N$ or $N^-$ on $U$ is exponentiated from the action of the universal enveloping algebras $\mrm U(\mf n)$ or $\mrm U(\mf n^-)$ (see Section \ref{subsec:locally-anal-reminder}).

\begin{lemm}\label{lemm:highest-wt-vectors}
Suppose that $U$ is an $L$-linear finite-dimensional $\mf g$-simple locally analytic representation of $G$. Then the following is true.
\begin{enumerate}
\item $U$ is irreducible as a $G$-representation.
\item If $U \in \mc O^G$ then $\dim_L U^N = 1 = \dim_L U^{N^-}$ and $U$ is generated as a $L[G]$-module by any non-zero vector in either $U^N$ or $U^{N^-}$.
\end{enumerate}
\end{lemm}
\begin{proof}
Since $U$ is finite-dimensional, any subspace $U' \subset U$ is closed. Thus if $U'$ were also $G$-stable, then it would become a subspace stable under the action of $\mrm U(\mf g_L)$ as well. Since $U$ is assumed to be irreducible as a module over $\mrm U(\mf g_L)$, we see that $U' = U$ or $U=0$ and this proves (1). 

For the second point, by symmetry it is enough to prove $\dim_L U^N = 1$ (the second claim is implied by the first claim and part (1) of the lemma). Since $\mf t_L$ acts on $U$ as a sum of characters, $U$ has maximal vectors for the weight action. Thus $U$ is a highest weight module since it is $\mf g$-simple. In particular, \cite[Theorem 1.2(c)]{hum} implies that $\dim_L U[\mf n_L] = 1$. Since $G$ acts on $U$ $L$-linearly, $U[\mf n]$ is also one-dimensional. The action of $N$ on $U$ is exponentiated from the action of $\mrm U(\mf n)$ (Proposition \ref{prop:exponentiate-action}) and so we deduce that $\dim_L U^N \geq 1$. On the other hand, the action of $\mrm U(\mf n_L)$ is the $L$-linearization of the differential action of $N$ as well, and thus $\dim_L U^N \leq \dim _L U[\mf n] = 1$.
\end{proof}

If $U$ is a $\mf g$-simple object in $\mc O^G$ then we define a highest weight vector with respect to $B$ to be any non-zero vector $u^+ \in U^N$. It is unique up to scalar in $L^\times$ by Lemma \ref{lemm:highest-wt-vectors}(2). Since $B$ normalizes $N$, $B$ acts on $U^N$ through a locally analytic character $\chi: T \ra L^\times$. We call $\chi$ the highest weight of $U$. We write $\Phi_{/L}^{\vee}$ for the co-characters of $\mbf T_{/L}$ and $\Lambda^+_{\mf g_L}$ is defined in Section \ref{subsec:verma-reminder}.

\begin{lemm}
Suppose that $U$ is a $\mf g$-simple object in $\mc O^G$ with highest weight $\chi$. Then $d\chi \in \Lambda_{\mf g_L}^+$, i.e. $\langle d\chi,\alpha^\vee \rangle \geq 0$ for all $\alpha^\vee \in \Phi^{\vee}_{/L}$.
\end{lemm}
\begin{proof}
This follows from the structure of simple $\mrm U(\mf g_L)$-modules (see \cite[Section 1.6]{hum}).
\end{proof}

\begin{prop}\label{prop:highest-weight-vectors-unique}
If $U$ and $U'$ are two $\mf g$-simple objects in $\mc O^G$ with the same highest weight character then $U\simeq U'$ as  $G$-representations.
\end{prop}
\begin{proof}
Write $\chi: T \ra L^\times$ for the common highest weight character. Since $U$ and $U'$ have the same infinitesimal highest weight, they are isomorphic as $\mrm U(\mf g_L)$-modules. Explicitly, a $\mrm U(\mf g_L)$-equivariant isomorphism $f:U \ra U'$ is obtained by choosing highest weight vectors $u^+ \in U$ and $u'^+ \in U'$ and setting $f(u^+) = u'^+$.  However, since $T$ acts on either $u^+$ or $u'^+$ via $\chi$, we see that $f$ is also $T$-equivariant. Indeed, if $X \in \mrm U(\mf g_L)$ then $f(Xu^+) = Xu'^+$, and we have
\begin{equation*}
f(t\cdot Xu^+) = f(\Ad(t)X \cdot tu^+) = \Ad(t) X f(tu^+) = \Ad(t)X\cdot  \chi(t) u'^+ = t\cdot Xu'^+.
\end{equation*}
Thus $f$ is a $(\mf g,T)$-equivariant isomorphism $f: U \simeq U'$. Since $U$ and $U'$ are in $\mc O^G$, Proposition \ref{prop:homs-ok} implies that $f$ is $G$-equivariant.
\end{proof}

\begin{coro}\label{coro:char-unique}
The association $U \mapsto U^N$ from $L$-linear finite-dimensional $\mf g$-simple objects in $\mc O^G$ to locally analytic characters $\chi: T \ra L^\times$ is injective.
\end{coro}

\begin{coro}\label{coro:isomorphism-transpose-dual}
Suppose that $U$ is a $\mf g$-simple object in $\mc O^G$. Then $U\simeq U^{\ast,\tau}$. The isomorphism is unique up to a scalar.
\end{coro}
\begin{proof}
First note that if $U$ satisfies the hypotheses of the corollary then so does $U^{\ast,\tau}$. Thus by Corollary \ref{coro:char-unique} it suffices to show that the action of $T$ on the one-dimensional spaces $U^N$ and $(U^{\ast,\tau})^N$ are the same. (This also implies the isomorphism is unique up to a scalar.)

Write $u^+ \in U^N$ for a highest weight vector, and let $u_1,\dotsc,u_n$ be an $L$-linear basis of $U$ with $u_1 = u^+$. Write $\lambda^+$ for the vector in $U^{\ast}$ dual to $u^+$ with respect to the basis $(u_i)$, i.e. $\lambda^+(u^+) = 1$ and $\lambda^+(u_i) = 0$ if $i > 1$. Note that $U$ is a weight module for $\mf t_L$, the action of $\mrm U(\mf n^-_L)$ on $U$ lowers weights and thus $\mrm U(\mf n^-_L)(U) \subset Lu_2 + \dotsb + Lu_n$. In particular, if $u \in U$ and $n^- \in N^-$ then Proposition \ref{prop:exponentiate-action} implies that $n^-\cdot u - u \in Lu_2+\dotsb + Lu_n$. We apply this as follows:\ if $n \in N$ then ${}^\tau \! n \in N^-$, and so we get the second equality in
\begin{equation*}
(n\cdot \lambda^+)(u) = \lambda^+({}^{\tau}\! n\cdot u) = \lambda^+(u).
\end{equation*}
This shows that $\lambda^+ \in (U^{\ast,\tau})^N$ is the unique, up to scalar, highest weight vector with respect to $B$. On the other hand, it is easy to check that $T$ acts on $\lambda^+$ via $\chi$ because $T$ is fixed point-by-point by the transpose map.
\end{proof}

\begin{exam}\label{exam:include_loc_algebraics}
If $U$ is an irreducible finite-dimensional $\mathbf Q_p$-algebraic representation of $G$ then $U$ is $\mf g$-simple. More generally, if $U$ is a locally $\mathbf Q_p$-algebraic representation which is absolutely irreducible then $U$ is $\mf g$-simple. Indeed, if $U$ is such a representation then a theorem of Prasad \cite[Theorem 1, Appendix]{st-Ugfinite} implies that $U \simeq U' \otimes \pi$ where $U'$ is an irreducible algebraic representation of $G$ and $\pi$ is an absolutely irreducible smooth representation of $G$. But since $G$ is the points of a split group $\mathbf G$, any absolutely irreducible, finite-dimensional and smooth $\pi$ is necessarily one-dimensional. Thus the $U\simeq U'$ as $\mrm U(\mf g)$-modules.
\end{exam}

\begin{rema}
Not every irreducible finite-dimensional locally analytic representation of $G$ is $\mf g$-simple. For example, see \cite[pg. 120]{st-Ugfinite}.
\end{rema}

\subsection{Internal duals of some generalized Verma modules}\label{subs:internal}
Suppose $U$ is an $L$-linear finite-dimensional locally analytic representation of $L_P$. Then we write $U^{\iota\tau}$ for the same underlying vector space $U$ equipped instead with the action of $L_P$ given by
\begin{equation*}
g\cdot_{\iota \tau} u = ({}^{\tau}\! g)^{-1}\cdot u = {}^{\tau}\! (g^{-1})\cdot u.
\end{equation*}
Since $\iota$ and $\tau$ are algebraic maps of $L_P \ra L_P$ this defines a new $L$-linear finite-dimensional locally analytic representation of $L_P$. The following lemma does not require the action of $\mf t_L$ on $U$ to be diagonalizable (compare with Remark \ref{rema:remove-diagonal-hyp}). 

\begin{lemm}\label{lemma:internal-versus-opposite}
If $U$ is an $L$-linear finite-dimensional locally analytic representation of $L_P$ then there is a natural $(\mf g,P)$-equivariant isomorphism $M_{\mf p_L}(U)^{\vee} \simeq M_{\mf p^-_L}(U^{\iota\tau})^-$.
\end{lemm}
\begin{proof}
This is similar to the proof of Lemma \ref{lemma:isomorphism-opposite-dual}. Consider the map
\begin{align*}
\Hom_{L}(\mrm U(\mf g_L)\otimes_{\mrm U(\mf p_L)} U, L)^\tau &\ra \Hom_{L}(\mrm U(\mf g_L)\otimes_{\mrm U(\mf p^-_L)} U^{\iota\tau}, L)^\iota\\
f &\mapsto \Psi_f
\end{align*}
given by $\Psi_f(X\otimes u) := f(d\iota d\tau (X)\otimes u)$. One must check that this is a well-defined isomorphism and equivariant for the structures of $\mrm U(\mf g_L)$-module and $L_P$-representation on each side. This being done, the result of the lemma follows by applying $(-)[{\mf n_{P,L}^\infty}]$. We check $\Psi_f$ is well-defined and $f \mapsto \Psi_f$ is $L_P$-equivariant, leaving the rest to the reader.

Let us check that $\Psi_f$ is well-defined for each $f$. The switch in the tensor products means we need to show that if $Y \in \mrm U(\mf p^-_L)$ then $\Psi_f(Y\otimes u) \overset{?}{=} \Psi_f(1\otimes Y\cdot_{\iota\tau} u)$, where the action on $u$ on the right-hand side is given by viewing $u \in U^{\iota\tau}$. But, if $Y \in \mrm U(\mf p^-_L)$ then $d\tau(Y) \in \mrm U(\mf p_L)$ and thus $d\iota d\tau(Y) \in \mrm U(\mf p_L)$ as well. We deduce that
\begin{align*}
\Psi_f(1\otimes Y\cdot_{\iota\tau} u) = f(1 \otimes d\iota d\tau(Y)\cdot u) = f(d\iota d\tau(Y) \otimes u) = \Psi_f(Y\otimes u),
\end{align*}
as desired. 

We now show $f \mapsto \Psi_f$ is equivariant for the action of $L_P$. If $g \in L_P$, $X \in \mrm U(\mf g_L)$, $u \in U$ and $f \in \Hom_{L}(\mrm U(\mf g_L)\otimes_{\mrm U(\mf p_L)} U, L)^\tau$ then
\begin{equation*}
\Psi_{g \cdot f}(X\otimes u) = (g\cdot f)(d\iota d\tau(X) \otimes u) = f(\Ad({}^{\tau}\! g) d\iota d\tau(X)\otimes {}^{\tau}\! g\cdot u)
\end{equation*}
whereas $\Psi_f \in \Hom_{L}(\mrm U(\mf g_L)\otimes_{\mrm U(\mf p^-_L)} U^{\iota\tau}, L)^\iota$ and so
\begin{equation*}
(g\cdot \Psi_f)(X\otimes u) = \Psi_f(g^{-1}(X\otimes u)) = \Psi_f(\Ad(g^{-1})X\otimes {}^{\tau}\! g\cdot u) = f(d\iota d\tau(\Ad(g^{-1} X)) \otimes {}^{\tau}\! g\cdot u).
\end{equation*}
(The third equality is because $g^{-1}$ acts on $u \in U^{\iota \tau}$ via ${}^\tau\! g = ({}^\tau\! (g^{-1}))^{-1}$.) We conclude that $\Psi_{g\cdot f}(X\otimes u) = (g\cdot \Psi_f)(X\otimes u)$ using the relation $d\iota d \tau (\Ad(g^{-1})X) = \Ad({}^{\tau}\! g)d\iota d \tau(X)$ (see Note \ref{footnote:trans-inv-relation} on page \pageref{footnote:trans-inv-relation}).
\end{proof}

\begin{rema}
For the next result, we apply Section \ref{subsec:digression}  with $\mbf L_{\mbf P}$ being the split connected reductive group.  We will use the notation $\mc O^{L_P}$ as in Section \ref{subsec:digression}, so objects in $\mc O^{L_P}$ are in particular $L$-linear finite-dimensional locally analytic representations of $L_P$ on which $\mf t_L$ acts diagonalizably. Since $\mbf L_{\mbf P}$ is a parabolic only when $\mbf P = \mbf G$, in which case either usage of $\mc O^{L_P} = \mc O^{G}$ is the same, there should be no confusion.
\end{rema}

\begin{prop}\label{prop:poly-internal-dual}
If $U$ is a $\mf l_P$-simple element of $\mc O^{L_P}$ then
\begin{equation*}
C^{\pol}(N_P,U) \simeq M_{\mf p_L}(U)^{\vee}
\end{equation*}
in the category $\mc O^P$. The isomorphism is canonical up to scalar.
\end{prop}
\begin{proof}
First, Corollary \ref{coro:isomorphism-transpose-dual} implies $U \simeq U^{\ast,\tau}$. On the other hand, we clearly have $(U^{\ast,\iota})^{\iota\tau} = U^{\ast,\tau}$ and thus Proposition \ref{prop:polynomia-opposite-dual} and Lemma \ref{lemma:internal-versus-opposite} combine to give an isomorphism
\begin{equation*}
C^{\pol}(N_P,U)\simeq M_{\mf p^-_L}(U^{\ast,\iota})^- \simeq M_{\mf p_L}(U^{\ast,\tau})^\vee \simeq M_{\mf p_L}(U)^\vee.
\end{equation*}
The final map canonical up to a scalar (by Corollary \ref{coro:isomorphism-transpose-dual}).
\end{proof}

Having realized the internal dual of $M_{\mf p_L}(U)$ concretely, we can now define an explicit $(\mf g, P)$-equivariant morphism $\alpha _U: M_{\mf{p}_L}(U) \ra M_{\mf{p}_L}(U) ^{\vee}$ which generalizes the corresponding map constructed classically in the Verma module theory \cite[Theorem 3.3(c)]{hum} (at least when $U$ is an $\mf l_P$-simple object in $\mc O^{L_P}$; this is also the context of the classical theory (compare with \cite[Observation (1), Section 9.8]{hum})).

We continue to assume that $U$ is an $\mf l_P$-simple object in $\mc O^{L_P}$. In order to define $\alpha_U$ it is enough, by Proposition \ref{prop:poly-internal-dual}, to define a $(\mf g,P)$-equivariant map
\begin{equation}\label{eqn:evaluate}
\alpha_U:M_{\mf{p}_L}(U) \ra C^{\pol}(N_P,U).
\end{equation}
Since $M_{\mf p_L}(U)$ is generated over $\mrm U(\mf g_L)$ by $U$, it is enough to define a $P$-equivariant map $U \ra C^{\pol}(N_P,U)$. But, if $u \in U$ then the function $f_u(n) = u$ on $N_P$ is clearly a polynomial $U$-valued function. Moreover, $u \mapsto f_u$ is easily checked to be $P$-equivariant.\footnote{Remember that the $P$-action on $U$ is inflated via the quotient $P \twoheadrightarrow L_P$.}

\begin{prop}\label{prop:irred-objects} If $U$ is an $\mf l_P$-simple object in $\mc O^{L_P}$ then the following conslusions hold.
\begin{enumerate}
\item The map $\alpha_U : M_{\mf p_L}(U) \ra M_{\mf p_L}(U)^{\vee}$ is unique in $\mc O^P$ up to a scalar.
\item The image of $\alpha _U$ is an irreducible object $L(U)$ in $\mc O^P$.
\item Each of $M_{\mf p_L}(U)$ and $M_{\mf p_L}(U)^{\vee}$ are finite length in $\mc O^P$ and $L(U)$ appears exactly once as a simple subquotient of $M_{\mf{p}_L}(U)$ or $M_{\mf{p}_L}(U)^{\vee}$.
\item The objects $L(U)$ in $\mc O^P$ are internal self-dual.
\item If $\chi$ is the highest weight of $U$ with respect to $B$ then there is a surjective morphism $M_{\mf b_L}(\chi) \twoheadrightarrow M_{\mf p_L}(U)$ in the category $\mc O^B$.
\end{enumerate}
\end{prop}
\begin{proof}
There is actually only one such map $\alpha_U$ in the category $\mc O^{\mf p_L}$ up to scalar (see \cite[Theorem 3.3(c)]{hum} and point (1) on \cite[p. 198]{hum}). This proves (1). The claims (2) and (3) also follow immediately from the analogous statements in $\mc O^{\mf p_L}$. For example, the image of $\alpha_U$ is known to be an irreducible $\mrm U(\mf g_L)$-module already (and thus irreducible in $\mc O^P$) \cite[Theorem 3.3(c)]{hum}. For (3), see Proposition \ref{prop:verma-facts}.

Let us prove (4). First, the dual map $\alpha_U^\vee: M_{\mf p_L}(U) \ra M_{\mf p_L}(U)^\vee$ is equal to $\alpha_U$ up to a scalar by part (1). Thus the image of $\alpha_U^\vee$ is equal to the image of $\alpha_U$, which is $L(U)$ by part (2). On the other hand, the image of $\alpha_U^\vee$ is also $L(U)^\vee$ by Proposition \ref{prop:duality-exact}. Thus $L(U) \simeq L(U)^\vee$.

We finish by proving (5). Let $u^+ \in U$ be the highest weight vector on which $T$ acts by $\chi$. Let $e_{\chi}^+$ be the basis of the one-dimensional $L$-vector space on which $T$ acts by $\chi$. Then we consider the $\mrm U(\mf g_L)$-equivariant map $M_{\mf b_L}(\chi) \ra M_{\mf p_L}(U)$ given by $1\otimes e_\chi^+ \mapsto 1 \otimes u^+$. This is $(\mf g,T)$-equivariant and thus also $(\mf g,B)$-equivariant by Proposition \ref{prop:homs-ok}(2). Furthermore, it is clearly surjective since $U$ is generated as a $\mrm U(\mf l_{P,L})$-module by $u^+$, and thus $M_{\mf p_L}(U)$ is generated as a $\mrm U(\mf g_L)$-module by $u^+$ as well.
\end{proof}

\subsection{Composition series}\label{subsec:verma2}
In this section we determine more precise information about the composition factors in $\mc O^P$ of $M_{\mf p_L}(U)$ and $M_{\mf p_L}(U)^{\vee}$ when $U$ is an $\mf l_P$-simple object in $\mc O^{L_P}$. If $U$ is such a representation then it has a highest weight $\chi: T \ra L^\times$, and it is a quotient of $M_{\mf b_L}(\chi)$ in the category $\mc O^B$ by Proposition \ref{prop:irred-objects}(5). By Proposition \ref{prop:homs-ok} and Corollary \ref{coro:fully-faithful}, the constituents of $M_{\mf p_L}(U)$ in $\mc O^P$ are among the constituents of $M_{\mf b_L}(\chi)$ in $\mc O^B$. Thus our focus begins with the case $\mbf P = \mbf B$. 

Let us discuss the notion of strongly linked weights. Recall we write $\Phi_{/L}$ for the roots of $\mbf T_{/L}$. Suppose that $\alpha \in \Phi_{/L}$ and denote by $\alpha$ also the corresponding root of $\mf t_L$. If $\lambda \in \mf t_L^\ast$ is a weight then we get the composition $\lambda\circ \alpha^\vee : L \rightarrow L$. We say that $\lambda$ is $\alpha$-integral if $\langle \lambda, \alpha^\vee\rangle := (\lambda\circ \alpha^\vee) (1)$ is an integer, and $\lambda$ is called $\alpha$-dominant if $\lambda$ is $\alpha$-integral and $\langle \lambda, \alpha^\vee \rangle$ is a non-negative integer. The Weyl group of $\mbf T_{/L}$ acts on the weights $\mf t_L^{\ast}$ via the ``dot action'', $w\cdot \lambda = w(\lambda + \rho_0) - \rho_0$ for all $w$ in the Weyl group, and where $\rho_0$ is the half-sum of the positive roots of $\mbf T_{/L}$. One then has the notion of strongly linked weights \cite[Section 5.1]{hum}. Namely, if $\mu$ and $\lambda$ are two weights of $\mf t_L$ then we write $\mu\uparrow \lambda$ if there exists a root $\alpha$ positive with respect to $\mbf B_{/L}$ for which $\lambda$ is $\alpha$-dominant and $\mu = s_{\alpha}\cdot \lambda$, where $s_{\alpha}$ is the simple reflection in the Weyl group of $\mbf T_{/L}$ corresponding to $\alpha$. More generally we say that $\mu$ is strongly linked to $\lambda$ if $\mu$ is connected to $\lambda$ by a sequence of $\uparrow$-links.

We adapt the notion of strong linkage to locally analytic characters of $T$. Suppose that $\chi: T \ra L^\times$ is a locally analytic character. Then, it induces a $\mathbf Q_p$-linear map $d\chi: \mathfrak t \ra L$ and thus an $L$-linear map $d\chi : \mathfrak t_L \ra L$ by extension of scalars. If $\alpha \in \Phi_{/L}$ we call $\chi$ locally $\alpha$-integral (resp. -dominant) if $d\chi$ is $\alpha$-integral (resp. -dominant). 

A priori, it is not clear how to extend the dot action to locally analytic characters, so let us start with two observations. First, if $\lambda$ is a weight of $\mf t_L$ and $\alpha \in \Phi_{/L}$ then
\begin{equation*}
s_{\alpha}\cdot \lambda - \lambda = s_{\alpha}(\lambda + {\rho_0}) - (\lambda+{\rho_0}) = -\langle \lambda + {\rho_0},\alpha^\vee\rangle\alpha.
\end{equation*}
Here we have used the explicit description of the action of simple reflections on weights. In particular, the weight $s_{\alpha}\cdot \lambda-\lambda$ depends only on $\langle \lambda + {\rho_0},\alpha^\vee\rangle$. 

Second, if $\eta: \mbf T_{/L} \ra {\mathbf G_m}_{/L}$ is an algebraic character then it induces a ($\mathbf Q_p$-)algebraic character $\eta: T \ra L^\times$ via the composition
\begin{equation*}
T = \mathbf T(K) = (\Res_{K/\mathbf Q_p} \mathbf T)(\mathbf Q_p) \subset (\Res_{K/\mathbf Q_p} \mathbf T\times_{\mathbf Q_p} L)(L) \overset{\eta}{\ra} L^\times.
\end{equation*}
A $\mbf Q_p$-algebraic character of $T$ is locally analytic, and so in particular each $\alpha \in \Phi_{/L}$ defines a natural locally analytic character of $T$.
\begin{rema}
It is worth mentioning that if $\alpha \in \Phi_{/L}$ then the induced character $\alpha: T \rightarrow L^\times$ is of the form
\begin{equation*}
T \overset{\beta}{\longrightarrow} K^\times \overset{\sigma}{\longrightarrow} L^\times
\end{equation*}
where $\beta$ is a root of $\mbf T$ and $\sigma \in \mc S$.
\end{rema}

\begin{defi}\label{defi:dot-action-integrality}
If $\alpha \in \Phi_{/L}$ and $\chi$ is a locally $\alpha$-integral character, we set $\eta =  \alpha^{-\langle d\chi + {\rho_0},\alpha^\vee\rangle}$ and then define $s_{\alpha}\cdot \chi := \chi \eta$.
\end{defi}
(Here we have used the exponential notation to emphasize the $\alpha$ is being viewed as an $L$-valued character of $T$.)

\begin{lemm}\label{lemm:uniqueness-dotaction}
Suppose that $\alpha \in \Phi_{/L}$  and $\chi$ is a locally $\alpha$-integral character. Then $s_{\alpha}\cdot \chi$ is the unique locally analytic character $\chi': T \ra L^\times$ such that 
\begin{enumerate}
\item $d\chi' = s_{\alpha}\cdot d\chi$ and 
\item $\chi'\chi^{-1}$ is $\mathbf Q_p$-algebraic.
\end{enumerate}
\end{lemm}
\begin{proof}
It is clear that $d(s_{\alpha}\cdot \chi) = s_{\alpha}\cdot d\chi$ and thus $\chi'= s_{\alpha}\cdot \chi$ satisfies the two assertions. If $\chi'\chi^{-1} = \eta$ is $\mathbf Q_p$-algebraic and $d\chi' = s_{\alpha}\cdot d\chi$ then we see immediately that $d\eta = -\langle d\chi + {\rho_0},\alpha^\vee\rangle \alpha$. Since both $\eta$ and $\alpha$ are $\mathbf Q_p$-algebraic, the equality implies that $\eta = \alpha^{-\langle d\chi + {\rho_0},\alpha^\vee\rangle}$ and we are done.
\end{proof}
 
If $\chi$ is a locally $\alpha$-integral character and it happens that $s_{\alpha}\cdot \chi$ is locally $\alpha'$-integral then we could form the iterated character $s_{\alpha'}s_{\alpha}\cdot \chi$ in the obvious way. Write $\Phi^+_{/L}$ for the elements of $\Phi_{/L}$ positive with respect to $\mbf B_{/L}$.

\begin{defi}\label{defi:strongly-linked}
Let $\chi$ and $\chi'$ be two locally analytic characters of $T$. We write $\chi' \uparrow \chi$ if $\chi' = \chi$ or there exists an $\alpha \in \Phi^+_{/L}$  such that $\chi$ is locally $\alpha$-dominant and $s_{\alpha} \cdot \chi = \chi'$. We say that $\chi'$ is strongly linked to $\chi$ if there exists a sequence $\alpha_1,\dotsc,\alpha_r \in \Phi^+_{/L}$
\begin{equation*}
\chi' = (s_{{\alpha_1}}\dotsb s_{{\alpha_r}}) \cdot \chi \uparrow  (s_{{\alpha_2}}\dotsb s_{{\alpha_r}})\cdot \chi \uparrow \dotsb \uparrow  s_{{\alpha_r}} \cdot \chi \uparrow \chi.
\end{equation*}
\end{defi}
We have the following characterization of strongly linked characters.
\begin{lemm}\label{lemm:characterization-strongly-linked}
Suppose that $\chi$ and $\chi'$ are two locally analytic characters of $T$. Then $\chi'$ is strongly linked to $\chi$ if and only if $d\chi'$ is strongly linked to $d\chi$ and $\chi'\chi^{-1}$ is $\mathbf Q_p$-algebraic.
\end{lemm}
\begin{proof}
Since $\chi$ is locally $\alpha$-dominant if and only if $d\chi$ is $\alpha$-dominant, the claim is clear by Lemma \ref{lemm:uniqueness-dotaction} and induction on the number of $\uparrow$ in Definition \ref{defi:strongly-linked}.
\end{proof}

\begin{lemm}\label{lemm:equivariants}
If $\chi$ and $\chi'$ are two locally analytic characters of $T$ which are strongly linked and $f: M_{\mf b_L}(\chi') \ra M_{\mf b_L}(\chi)$ is $\mrm U(\mf g_L)$-equivariant then $f$ is $T$-equivariant.
\end{lemm}
\begin{proof}
This is an explicit computation. If $f = 0$ then the result is trivial. If $f\neq 0$ then by \cite[Theorem 4.2]{hum} the map $f$ is injective. Thus we suppress it from the notation and just view $M_{\mf b_L}(\chi')$ as a $\mrm U(\mf g_L)$-submodule of $M_{\mf b_L}(\chi)$. We need to show that it is a $T$-submodule as well.

Write $e_{\chi}^+$ (resp.\ $e_{\chi'}^+$) for the highest weight vector in $M_{\mf b_L}(\chi)$ (resp.\ $M_{\mf b_L}(\chi')$). Since $M_{\mf b_L}(\chi)$ is generated by $e_{\chi}^+$ over $\mrm U(\mf n_{B,L}^-)$ we can write $e_{\chi'}^+ = Xe_{\chi}^+$ for some $X \in \mrm U(\mf n_{B,L}^-)$. If $Y \in \mf t_L$ then
\begin{equation*}
YXe_\chi^+ = (XY + \ad(Y)(X))e_\chi^+ = (d\chi(Y)X + \ad(Y)(X))e_\chi^+
\end{equation*}
and
\begin{equation*}
YXe_\chi^+ = Ye_{\chi'}^+ = d\chi'(Y)e_{\chi'}^+ = d\chi'(Y)X e_{\chi}^+.
\end{equation*}
Thus $\ad(Y)(X) - (d\chi'(Y)-d\chi(Y))X$ annihilates $e_\chi^+$. But $X, \ad(Y)(X) \in \mrm U(\mf n_{B,L}^-)$ and $M_{\mf b_L}(\chi) = \mrm U(\mf n_{B,L}^-) \otimes_L Le_\chi^+$. Thus we deduce that $\ad(Y)(X) = (d\chi'(Y)-d\chi(Y))X$.

But now the $\ad$-action of $\mf t_L$ and the $\mathbf Q_p$-algebraic adjoint action of $T$ on $\mrm U(\mf n_{B,L}^-)$ have the same eigenvectors. We deduce that $X \in \mrm U(\mf n_{B,L}^-)$ is an eigenvector for the adjoint action of $T$, with a $\mathbf Q_p$-algebraic eigensystem we write $\psi_X$. To finish the proof, we just need to show that $\psi_X = \chi' \chi^{-1}$.  To see that, we note that since $\chi'$ is strongly linked to $\chi$, the character $\chi' \chi^{-1}$ is $\mbf{Q}_p$-algebraic (Lemma \ref{lemm:characterization-strongly-linked}) and thus the equality $\psi_X = \chi'\chi^{-1}$ is equivalent to $d(\psi_X) = d\chi' - d\chi$. But this equality follows from our computation of $\ad(Y)(X)$ above.
\end{proof}

\begin{prop}\label{prop:verma-compositions}
If $\chi$ and $\chi'$ are two locally analytic characters of $T$ then $L(\chi')$ is a composition factor of $M_{\mf b_L}(\chi)$ in $\mc O^B$ if and only if $\chi'$ is strongly linked to $\chi$. Moreover, all the composition factors of $M_{\mf b_L}(\chi)$ in $\mc O^B$ are of this form.
\end{prop}
\begin{rema}
Proposition \ref{prop:verma-compositions} is a more precise version of Theorem \ref{theo:constituents}(a) below since we specify that {\em if} $\chi'$ is strongly linked to $\chi$ {\em then} $L(\chi')$ actually appears as a composition factor in $M_{\mf b}(\chi)$.
\end{rema}

\begin{proof}[Proof of Proposition \ref{prop:verma-compositions}]
Suppose first that $\chi'$ is strongly linked to $\chi$. By BGG reciprocity \cite[Theorem 5.1(a)]{hum} $L(\chi')$ is a composition factor of $M_{\mf b_L}(\chi)$ as a $\mrm U(\mf g_L)$-module, and there exists a non-zero $\mrm U(\mf g_L)$-equivariant morphism $M_{\mf b_L}(\chi') \ra M_{\mf b_L}(\chi)$. By Proposition \ref{prop:homs-ok}(2) and Lemma \ref{lemm:equivariants},  the morphism is automatically a morphism in the category $\mc O^B$ and so $L(\chi')$ is a subquotient of $M_{\mf b_L}(\chi)$ in $\mc O^B$.

Now suppose that $\Lambda$ is an irreducible subquotient of $M_{\mf b}(\chi)$ in the category $\mc O^B$.  We note that $T$ acts diagonalizably on $M_{\mf b_L}(\chi)$ (because $T$ acts semi-simply on $\mrm U(\mf n_{B,L}^-)$ and $M_{\mf b}(\chi)$ is generated over $\mrm U(\mf n_{B,L}^-)$ by an eigenvector for $T$). Choose a highest weight $\mu$ for $\Lambda$, i.e. $\mrm U(\mf n_{B,L}) \cdot \Lambda_{\mu} = (0)$ where $\Lambda_\mu$ is the set of vectors $x \in \Lambda$ on which $\mf t_L$ acts by $\mu$. The $\mu$-eigenspace $\Lambda_{\mu}$ is $T$-stable, thus a sum of one-dimensional $T$-stable spaces, so we can choose a non-zero vector $v \in \Lambda_{\mu}$ which is an eigenvector for $T$. 

Since $\mf n_{B,L}$ annihilates $v$, $B$ acts on $v$ through a locally analytic character $\chi'$ (Proposition \ref{prop:homs-ok}(1)). But then $v$ defines a non-zero morphism $M_{\mf b_L}(\chi') \ra \Lambda$ in $\mc O^B$. Since $\Lambda$ is irreducible, it is an irreducible quotient of $M_{\mf b}(\chi')$. But then by Proposition \ref{prop:irred-objects} we have $\Lambda \simeq L(\chi')$. Finally, we need to check that $\chi'$ is strongly linked to $\chi$. On the level of the Lie algebra, $d\chi'$ is strongly linked to $d\chi$ by \cite[Theorem 5.2(b)]{hum}. Moreover, $\chi'$ appears as a $T$-eigensystem in $M_{\mf b_L}(\chi)$ and thus $\chi'\chi^{-1}$ is $\mathbf Q_p$-algebraic. We conclude by Lemma \ref{lemm:characterization-strongly-linked}.
\end{proof}

If $U$ is an $\mf l_P$-simple object in $\mc O^{L_P}$ with highest weight $\chi$ then $M_{\mf p_L}(U)$ is a quotient of $M_{\mf b_L}(\chi)$ in the category $\mc O^B$. Thus, by the full faithfulness of $\mc O^P \subset \mc O^B$ (see Proposition \ref{coro:fully-faithful}) now know the composition factors of $M_{\mf p_L}(U)$ are among the $L(\chi')$ with $\chi'$ strongly linked to $\chi$. However, this does not explicitly explain how to realize the $(\mf g,B)$-module $L(\chi')$ inside the category $\mc O^P$.

\begin{lemm}\label{lemma:L(chi)inOP}
Suppose that $\chi: T \ra L^\times$ is a locally analytic character and $L(\chi)$ is in $\mc O^P$. Then the following conclusions hold.
\begin{enumerate}
\item $d\chi \in \Lambda_{\mf p_L}^+$.
\item If $\bar e_\chi^+ \in L(\chi)$ is the highest weight vector for the action of $B$ then any $\mrm U(\mf p_L)$-stable subspace of $L(\chi)$ containing $\bar e_\chi^+$ is also $T$-stable.
\item The $\mrm U(\mf p_L)$-submodule generated by $\bar e_\chi^+$ inside $L(\chi)$ is finite-dimensional, $\mrm U(\mf p_L)$-irreducible, $\mrm U(\mf n_{P,L})$ acts trivially on it, and $\mf t_L$ acts diagonalizably.
\end{enumerate}
\end{lemm}
\begin{proof}
To prove (1) we just note that our hypothesis implies that $L(\chi)$ is in the BGG category $\mc O^{\mf p_L}$ as a $\mrm U(\mf g_L)$-module. Thus $d\chi \in \Lambda_{\mf p_L}^+$ by \cite[Proposition 9.3(e)]{hum}.

Proving (2) is a short calculation. Let $X \in \mrm U(\mf p_L)$ and $t \in T$. Then 
\begin{equation*}
t X \bar e_\chi^+ = (\Ad(t) X) \cdot te_\chi^+ = \chi(t)(\Ad(t)X)e_\chi^+.
\end{equation*}
Since $\Ad(t)$ stabilizes $\mrm U(\mf p_L)$, the right-hand side is in the $\mrm U(\mf p_L)$-module generated by $\bar e_\chi^+$.

We end by proving (3). By part (1) of this lemma, $d\chi \in \Lambda_{\mf p_L}^+$. Thus there exists a finite-dimensional simple $\mrm U(\mf l_{P,L})$-module $W_{d\chi}$ as in Section \ref{subsec:verma-reminder}. By Proposition \ref{prop:verma-facts}(5), $W_{d\chi}$ is exactly the $\mrm U(\mf p_L)$-submodule generated by $\bar e_\chi^+$ inside $L(\chi)$. In particular, it is finite-dimensional, $\mrm U(\mf p_L)$-irreducible, $\mrm U(\mf n_{P,L})$ acts trivially and $\mf t_L$ acts diagonalizably.
\end{proof}

\begin{theo}\label{theo:constituents}
If $U$ is an $\mf l_P$-simple object in $\mc O^{L_P}$ with highest weight $\chi$, then:
\begin{enumerate}
\item Every irreducible constituent of $M_{\mf p_L}(U)$ in $\mc O^P$ is of the form $L(\chi')$ for some $\chi'$ strongly linked to $\chi$ and $d\chi' \in \Lambda_{\mf p_L}^+$.
\item The list of irreducible constituents of $M_{\mf p_L}(U)$ and $M_{\mf p_L}(U)^{\vee}$ are the same, with multiplicity.
\end{enumerate}
\end{theo}

\begin{proof}
Part (2) follows from Proposition \ref{prop:duality-exact}, part (1) of this theorem and Proposition \ref{prop:irred-objects}(4). It remains to show part (1). We remark that we have already shown part (1) when $\mbf P = \mbf B$. Suppose that $\Lambda$ is an irreducible constituent of $M_{\mf p_L}(U)$ in $\mc O^P$. Then $\Lambda$ is also an irreducible object in $\mc O^B$ by Proposition \ref{prop:homs-ok}(3). Since $M_{\mf p_L}(U)$ is a quotient of $M_{\mf b_L}(\chi)$ in $\mc O^B$ we deduce from Proposition \ref{prop:verma-compositions} that $\Lambda \simeq L(\chi')$ for some $\chi'$ strongly linked to $\chi$. Since $L(\chi') \in \mc O^P$, we see $d\chi' \in \Lambda_{\mf p_L}^+$ from Lemma \ref{lemma:L(chi)inOP}(1).
\end{proof}
It is worth wondering about the explicit $(\mf g,P)$-structure on a constituent $L(\chi')$ appearing in Theorem \ref{theo:constituents}(a). We end with the following result.

\begin{prop}\label{prop:existence-findimirred}
Suppose that $\chi : T \ra L^\times$ is a locally analytic character and $L(\chi)$ is in $\mc O^P$. Then there exists an $\mf l_P$-simple object $U$ in $\mc O^{L_P}$ with highest weight $\chi$.
\end{prop}
\begin{proof}
Let $\bar e_\chi^+$ be the highest weight vector in $L(\chi)$. Let $U$ be the smallest $P$-subrepresentation of $L(\chi)$ containing $\bar e_\chi^+$. We claim that $U$ is the representation we are seeking.

First, $U$ is finite-dimensional because $L(\chi)$ is in $\mc O^P$ (see Definition \ref{defi:OP}(1)). Second,  Proposition \ref{prop:homs-ok}(3) and Lemma \ref{lemma:L(chi)inOP}(2) imply that $U$ is also the $\mrm U(\mf p_L)$-submodule generated by $\bar e_{\chi}^+$. By Lemma \ref{lemma:L(chi)inOP}(3) we deduce that $U$ is irreducible as a module over $\mrm U(\mf p_L)$, the action factors through the quotient $\mrm U(\mf p_L) \twoheadrightarrow \mrm U(\mf l_{P,L})$ and $\mf t_L$ acts diagonalizably on $U$. In particular, $U$ is irreducible as a $P$-representation as well. Furthermore, since $\mrm U(\mf n_{P,L})$ acts by zero on $U$ we see that $N_P$ acts trivially on $U$ by Proposition \ref{prop:exponentiate-action}. Thus the action of $P$ on $U$ factors through $L_P$.  Hence $U$ is an $\mf l_P$-simple object in $\mathcal O^{L_P}$. Its highest weight is $\chi$ by construction, finishing the proof.
\end{proof}

\section{The adjunction formula}\label{sec:adj-formula}

Our adjunction formula relates certain locally analytic principle series with the Emerton--Jacquet functor of a locally analytic representation of the group $G$.

\begin{defi}\label{defi:vsa}
Let $V$ be an $L$-linear locally analytic representation of $G$.
\begin{enumerate}
\item $V$ is called very strongly admissible if $V$ is a locally analytic admissible representation of $G$ and there exists a continuous $L$-linear and $G$-equivariant injection $V \hookrightarrow B$ where $B$ is a continuous admissible representation of $G$ on an $L$-Banach space.
\item $V$ is called f-$\mf p$-acyclic if $\Ext_{\mf p_L}^1(U,V) = (0)$ for all finite-dimensional $L$-linear locally analytic representations $U$ of $L_P$.
\end{enumerate}
\end{defi}

\begin{exam}\label{example:acyclic}
The second part of Definition \ref{defi:vsa} may be a little ad hoc, so we give an example. (The ``f'' in the definition refers to finite, as in finite-dimensional).

Suppose that $B$ is an admissible $L$-Banach representation of $G$, and there exists a compact open subgroup $H$ such that $B|_{H} \simeq \mc C^0(H,L)^{\oplus r}$ for some integer $r \geq 1$. Then the locally analytic vectors $V = B_{\an}$ clearly satisfy Definition \ref{defi:vsa}(1), and we claim it is also satisfies Definition \ref{defi:vsa}(2) (compare also with the proof of \cite[Proposition 6.3.3]{br3}). Indeed, the action of $\mf p_L$ is unaffected by replacing $G$ by $H$ so it suffices to show $\Ext^1_{\mf p_L}(U,\mc C^{\an}(G,L)) = (0)$ for all finite-dimensional locally analytic representations $U$ of $L_P$, when $G$ is compact (and thus $P$ and $L_P$ are also compact).

We make three general observations. First, if $W$ is any locally analytic representation of $P$ then $C^{\an}(P,W) \simeq C^{\an}(P,W_{\operatorname{triv}})$ where $W_{\operatorname{triv}}$ is the underlying vector space of $W$ equipped with its trivial action. Indeed, the isomorphism is given by taking a locally analytic function $f: P \rightarrow W$ to the locally analytic function $f': P \rightarrow W_{\operatorname{triv}}$ given by $f'(g) = g^{-1}(f(g))$.  Second, since $G$ is compact we have a topological isomorphism $G \simeq P \times G/P$, compatible with the natural actions of $P$, and thus also
\begin{equation*}
C^{\an}(G,L) \simeq C^{\an}(P,L) \widehat{\otimes}_L C^{\an}(G/P,L),
\end{equation*}
(\cite[Lemma A.1 and Proposition A.2]{st-duality}) as locally analytic representations of $P$ with $P$ acting trivially on the second tensorand. Finally, if $U$ is finite-dimensional then we have a canonical isomorphism
\begin{equation*}
\Ext^1_{\mathfrak p_L}(U,C^{\an}(G,L)) \simeq \Ext^1_{\mathfrak p_L}(L,U^{\ast}\otimes_L C^{\an}(G,L)).
\end{equation*}
We now finish the example. We want to show that the left-hand space in the previous equation vanishes. The space on the right-hand side is the Lie algebra cohomology $H^{i}(\mathfrak p_L, M)$ in degree $i=1$, where $M$ is the $\mathfrak p_L$-module
\begin{equation*}
M = U^{\ast}\otimes_L C^{\an}(G,L) \simeq C^{\an}(P,U^{\ast})\widehat{\otimes}_L C^{\an}(G/P,L).
\end{equation*}
We owe the isomorphism to our second observation above. Our first observation now implies that we may assume that $U = L$ is trivial. Thus we have reduced to showing that $H^1(\mathfrak p_L, C^{\an}(G,L)) = (0)$. We will show this vanishing is true in degrees $i > 0$ in fact.

First consider the case where $P = G$. Recall (\cite[Section 3]{st-duality}) that the cohomology $H^{i}(\mathfrak p_L,C^{\an}(P,L))$ is computed as the cohomology of the complex
\begin{equation}\label{eqn:differential-exactness}
\dotsb \rightarrow \Hom_L(\wedge^{i-1}\mathfrak p_L, C^{\an}(P,L)) \rightarrow \Hom_L(\wedge^{i}\mathfrak p_L, C^{\an}(P,L)) \rightarrow \dotsb
\end{equation}
(with the standard differentials). By (the proof of) \cite[Proposition 3.1]{st-duality}, this sequence is exact and the differentials are strict morphisms. Thus, \eqref{eqn:differential-exactness} remains exact after taking $-\widehat{\otimes}_L C^{\an}(G/P,L)$ (see \cite[Lemma 4.13]{schraen} for example, or compare with the proof of \cite[Proposition 2.1.23]{em6}). On the other hand, since $\wedge^j \mathfrak p_L$ is finite-dimensional for each $j$, the completed tensor product of \eqref{eqn:differential-exactness} with $-\widehat{\otimes}_L C^{\an}(G/P,L)$ gives us an exact sequence
\begin{equation}\label{eqn:second-exactsequence}
\dotsb \rightarrow \Hom_L(\wedge^{i-1}\mathfrak p_L, C^{\an}(G,L)) \rightarrow \Hom_L(\wedge^{i}\mathfrak p_L, C^{\an}(G,L)) \rightarrow \dotsb.
\end{equation}
Moreover, since $\mathfrak p_L$ acts trivially on $C^{\an}(G/P,L)$ the differential induced from \eqref{eqn:differential-exactness} and extending scalars is the same as the natural differential on \eqref{eqn:second-exactsequence} that is used to compute $H^i(\mathfrak p_L, C^{\an}(G,L))$. Since \eqref{eqn:second-exactsequence} is exact, we have shown what we want and this completes the example.
\end{exam}

\begin{rema}\label{rema-smooth}
The vanishing of $\Ext^1_{\mf p_L}(U,V)$ in Definition \ref{defi:vsa}(2) also implies it for $U$ replaced by $U\otimes_L \Pi$ where $\Pi$ is a smooth representation of $L_P$ (because of how $\Ext$ commutes with arbitrary direct sums in the first coordinate).
\end{rema}

To a very strongly admissible locally analytic representation $V$ of $G$ (but also more generally), Emerton has associated a locally analytic representation $J_P(V)$ of $L_P$ which we call the Emerton--Jacquet module of $V$  \cite{em2,em3}.

By  \cite[Theorem 0.3]{em2}, $J_P(-)$ is characterized as being right adjoint to the functor $U \mapsto \mc C^{\sm}_c(N_P,U)$ on $L$-linear locally analytic representations of $L_P$ (on compact type spaces). If $N^0_P \subset N_P$ is a fixed compact open subgroup then, by definition, the space $J_P(V)$ is closely related to the space of invariants $V^{N_P^0}$ \cite[Definition 3.4.5]{em2}. In fact, on $V^{N_P^0}$ there is a canonical action of the monoid $L_P^+ = \{t \in L_P \mid tN_P^0t^{-1} \subset N_P^0\}$ by
\begin{equation}\label{eqn:hecke-action}
\pi_t\cdot v := {1 \over (N^0_P:tN^0_Pt^{-1})} \sum_{a \in N^0_P/tN^0_Pt^{-1}} at\cdot v = \delta_P(t)\sum_{a \in N_P^0/tN^0_Pt^{-1}} at \cdot v,
\end{equation}
for each $t \in L_P^+$ and $v \in V^{N_P^0}$ (here $\delta_P$ is the modulus character of $P$). If $Z(L_P)$ is the center of $L_P$ then $J_P(V)$ is the ``finite slope'' part of $V^{N_P^0}$ with respect to the monoid $Z(L_P)^+ = Z(L_P) \cap L_P^+$ and the action of $L_P^+$ extends to an action of all $L_P$.

When $V$ is very strongly admissible, the locally analytic action of the center $Z(L_P)$ on $J_P(V)$ extends to an action of the space $\mc O(\widehat{Z(L_P)})$ of rigid analytic functions on the $p$-adic rigid space $\widehat{Z(L_P)}$ parameterizing locally analytic characters of $Z(L_P)$ (this is part of the content of \cite[Theorem 0.5]{em2}). If $\eta \in \widehat{Z(L_P)}$ is a character then the set of functions vanishing at $\eta$ defines a maximal ideal $\mf m_{\eta} \subset \mc O(\widehat{Z(L_P)})$. We write $J_P(V)^{Z(L_P)=\eta}$ for the closed $L_P$-stable subspace which is annihilated by $\mf m_{\eta}$, equivalently the closed subspace of $J_P(V)$ on which $Z(L_P)$ acts through the character $\eta$, equivalently the closed subspace $V^{N^0_P,Z(L_P)^+ = \eta}$ (\cite[Proposition 3.4.9]{em2}).

We are now ready to give the key definition. 

\begin{defi}\label{defi:badness}
Suppose that 
\begin{itemize}
\item $U$ is an $\mf l_P$-simple object in $\mathcal O^{L_P}$ with  highest weight $\chi$,
\item $\pi$ is an $L$-linear (admissible) smooth $L_P$-representation of finite length admitting a central character $\omega_{\pi}$,\footnote{We include the adjective admissible for emphasis, but if $\pi$ is smooth and has finite length then $\pi$ is automatically admissible. The irreducible case was proved by Jacquet \cite[Th\'eor\`eme 1]{jacquet} and the finite length case follows because admissible smooth representations are closed under extension.} and
\item $V$ is an $L$-linear very strongly admissible locally analytic representation of $G$.
\end{itemize}
Then, we say the pair $(U,\pi)$ is non-critical with respect to $V$ if $J_P(V)^{Z(L_P) = \chi'\omega_{\pi}} = (0)$ for every locally analytic character $\chi'\neq \chi$ of $T$ which is strongly linked to $\chi$ and $d\chi' \in \Lambda_{\mf p_L}^+$.
\end{defi}
\begin{rema}\label{remark:we-wish}
One would like to only qualify over those characters $\chi'$ such that $L(\chi')$ appears as a constituent of the generalized Verma module $M_{\mf p}(U)$ but we were unable to complete the proof of Theorem \ref{theo:adj-theo} under that hypothesis. When $\mbf P = \mbf B$, which is currently the most important case for applications, there is no ambiguity by Proposition \ref{prop:verma-compositions}.
\end{rema}

\begin{rema}\label{rema:non-critical-slope}
Since $V$ is assumed to be very strongly admissible, \cite[Lemma 4.4.2]{em2} provides a sufficient condition for $(U,\pi)$ to be non-critical, independent of $V$. Indeed, it is sufficient that the character $\chi\omega_\pi$ be of {\em non-critical slope} in the sense of \cite[Definition 4.4.3]{em2} (at least when $\chi$ is assumed to be locally algebraic).\footnote{There is a typo in \cite[Lemma 4.4.2]{em2}. The slope (in the sense of \cite{em2}) of the modulus character is $-2\rho$, not $-\rho$ (here $\rho$ is the half-sum of positive roots, what we are denoting $\rho_0$). Thus, $\rho$ should be replaced by $2\rho$ in the statement of \cite[Lemma 4.4.2]{em2}. Despite this, \cite[Definition 4.4.3]{em2} should remain unchanged.}
\end{rema}

Condition (2) in Definition \ref{defi:badness} implies a more natural looking condition.
\begin{lemm}\label{lemm:gives_old_cond}
Suppose that $U$, $\pi$ and $V$ are a triple as in Definition \ref{defi:badness} and $(U,\pi)$ is non-critical with respect to $V$. If $\chi' \neq \chi$ is a locally analytic character of $T$, $\chi'$ is strongly linked to $\chi$ with $d\chi \in \Lambda_{\mf p_L}^+$ and $U_{\chi '}$ is an $L$-linear finite-dimensional $\mrm U(\mf l_{P,L})$-simple representation of $L_P$ with the highest weight $\chi '$ then  $\Hom_{L_P}(U_{\chi'}\otimes \pi, J_P(V)) = (0)$.
\end{lemm}
\begin{proof}
If $f \in \Hom_{L_P}(U_{\chi'}\otimes \pi, J_P(V))$ then the image of $f$ lies in $J_P(V)^{Z(L_P)=\chi'\omega_\pi} =(0)$.
\end{proof}
We now state our main theorem.

\begin{theo}\label{theo:adj-theo}
Suppose that 
\begin{itemize}
\item $V$ is an $L$-linear very strongly admissible, f-$\mf p$-acyclic, representation of $G$,
\item $U$ is an $L$-linear finite-dimensional $\mf l_P$-simple object in $\mathcal O^{L_P}$ and 
\item $\pi$ is a finite length smooth representation of $L_P$ admitting a central character. 
\end{itemize}
If the pair $(U,\pi)$ is non-critical with respect to $V$ then there exists a canonical isomorphism
\begin{equation*}
\Hom_G\left(\Ind_{P^-}^{G}(U \otimes \pi(\delta_P^{-1}))^{\an}, V\right) \simeq \Hom_{L_P}\left(U\otimes \pi, J_P(V)\right).
\end{equation*}
\end{theo}
A version of this theorem in the case that $\mbf P = \mbf B$ and $\mbf G = {\GL_3}_{/\mbf Q_p}$ (though that assumption was not crucially used) was proven in \cite[Theorem 4.1]{bcho}. The proof we give here is slightly different, inspired by \cite{be}. 

Throughout the rest of this section we fix $U,\pi$ and $V$ as in the statement of Theorem \ref{theo:adj-theo}. We also set $U' = U\otimes_L \pi(\delta_P^{-1})$, which is a locally analytic representation of $L_P$ when equipped with its finest convex topology (which is the same as the inductive tensor product topology since $U$ and $\pi(\delta_P^{-1})$ are each respectively equipped with their finest convex topology.)

Let $\mc C_c^{\lp}(N_P,L)$ be the space of compactly supported $L$-valued locally polynomial functions on $N_P$. If $W$ is any locally analytic representation of $L_P$ then we denote by $\mc C_c^{\lp}(N_P,W)$ the $W$-valued locally polynomial functions $\mc C_c^{\lp}(N_P,L)\otimes_L W$ as in \cite[Definition 2.5.21]{em3}. Because of the natural open immersion $N_P \hookrightarrow G / P^-$, we can regard $\mc{C}_c^{\lp}(N_P, U' )$ as a $(\mf{g}, P)$-stable subspace of $\Ind _{P^-}^{G} (U' )^{\an}$. If $\mc C_c^{\sm}(N_P,U')$ denotes the space of smooth compactly supported functions then we have a $P$-equivariant map $\mc C_c^{\sm}(N_P,U') \hookrightarrow \mc{C}_c^{\lp}(N_P, U')$ which is a closed embedding (see \cite[Section 2.8]{em3}). We note now that since $U'$ is equipped with its finest convex topology, so is $\mc C_c^{\sm}(N_P,U')$ by definition (see \cite[Section 3.5]{em2}). The same is true if we replace $U'$ by just $\pi(\delta_P^{-1})$, and the same is true if we consider the locally analytic $(\mf g,P)$-module $\mrm U(\mf g _L)\otimes_{\mrm U(\mf p _L)} \mc C_c^{\sm}(N_P,U')$.

We consider the following diagram (which is commutative, as we will explain):
\begin{equation*}\label{eqn:big_diagram}
\renewcommand{\labelstyle}{\textstyle}
\xymatrix@R=3pc@C=1pc{
\Hom _{G} (\Ind _{P^{-}} ^{G} (U')^{\an}, V) \ar[r]^-{(1)} \ar[d]^-{\simeq}_-{\text{(a)}}  &  \Hom _{L_P} (U\otimes_L \pi, J_{P}(V) )\ar[d]_{\simeq}^{\text{(b)}} \\
\Hom _{(\mf{g},P)}(\mc C_c^{\lp}(N_P,U'), V) \ar[r] ^-{(2)} \ar[d]^-{\simeq}_-{\text{(c)}} & \Hom _{(\mf{g},P)}(\mrm U(\mf g _L)\otimes_{\mrm U(\mf p _L)} \mc C_c ^{\sm}(N_P,U'), V) \ar[d]^{\text{(d)}}_-{\simeq} \\
\Hom _{(\mf{g},P)}(M_{\mf p _L}(U)^{\vee} \otimes \mc{C} _c ^{\sm}(N_P, \pi(\delta_P^{-1})), V) \ar[r]^{(3)} \ar[d]_-{\text{(3a)}} & \Hom _{(\mf{g},P)}(M_{\mf p _L}(U) \otimes_L \mc{C}_c ^{\sm}(N_P,\pi(\delta_P^{-1})),V) \\
\Hom _{(\mf{g},P)}(L(U) \otimes \mc{C}_c ^{\sm}(N_P,\pi(\delta_P^{-1})), V) \ar[ur]_-{(3\text{b})}
}
\end{equation*}
The statement of Theorem \ref{theo:adj-theo} is that the map (1) is an isomorphism. Let us now explain all the identifications and maps.

The map (1) is obtained by applying the Emerton--Jacquet functor, together with the canonical $L_P$-equivariant embedding $U\otimes \pi \hookrightarrow J_P(\Ind_{P^-}^G(U')^{\an})$  given in \cite[Lemma 0.3]{em3}. The map (a) is the natural restriction morphism (it is an isomorphism by \cite[Corollary 4.3.3]{em3}). The map (2) is induced from the canonical $P$-equivariant inclusion $\mc C^{\sm}_c(N_P,U') \hookrightarrow \mc C^{\lp}_c(N_P,U')$, and the map (b) is (the inverse of) the adjunction theorem/definition of the Emerton--Jacquet functor (see  \cite[Theorem 3.5.6]{em2} and \cite[(0.17)]{em3}).  The upper square commutes by construction---the canonical inclusion $U\otimes \pi \hookrightarrow J_P(\Ind_{P^-}^G(U')^{\an})$ is induced from the equality $U\otimes \pi \simeq J_P(\mc C_c^{\sm}(N_P,U'))$ (\cite[Lemma 3.5.2]{em2}) together with applying the Emerton--Jacquet functor to the canonical inclusion $\mc C_c^{\sm}(N_P,U') \hookrightarrow \Ind_{P^-}^G(U')^{\an}$ (see the proof of \cite[Lemma 0.3]{em3} in \cite[Section 2.8]{em3}).

The map (d) is the $(\mf g,P)$-equivariant identification
\begin{align*}
\mrm U(\mf g _L) \otimes_{\mrm U(\mf p _L)} \mc C_c^{\sm}(N_P,U') &= \mrm U(\mf g _L)\otimes_{\mrm U(\mf p _L)} \left(U\otimes_L \mc C_c^{\sm}(N_P,\pi(\delta_P^{-1}))\right) \\
&= M_{\mf p _L}(U)\otimes_L \mc C_c^{\sm}(N_P,\pi(\delta_P^{-1})),
\end{align*}
which is a definition at each step. The map (c) is the identification $\mc C_c^{\lp}(N_P,U') \simeq \mc{C}^{\pol}(N_P,U) \otimes_L \mc{C}_c ^{\sm}(N_P,\pi(\delta_P^{-1}))$ together with Proposition \ref{prop:poly-internal-dual}. The map (3) is induced from the map
\begin{equation*}
\alpha _{U}: M_{\mf p _L}(U) \ra M_{\mf p _L}(U) ^{\vee}
\end{equation*}
defined in Section \ref{subs:internal}. The commutation of the middle square follows from the construction of $\alpha_U$ via \eqref{eqn:evaluate}. Finally, by Proposition \ref{prop:irred-objects}, $\alpha_U$ factors  as $M_{\mf p _L}(U) \twoheadrightarrow L(U) \hookrightarrow M_{\mf p _L}(U)^{\vee}$, which gives the maps (3a) and (3b).

We will now prove Theorem \ref{theo:adj-theo} by showing that the map (3) is an isomorphism. In turn we will show separately that the maps (3a) and (3b) are isomorphisms.
\begin{prop}\label{prop:iso-and-inj}
The map (3b) is an isomorphism and (3a) is injective.
\end{prop}
\begin{proof}
Let $M' = \ker(M_{\mf p _L}(U) \ra L(U))$. $M'$ lies in the category $\mc O^P$. Since $M'$ is a subobject of $M_{\mf p_L}(U)$, $M'$ has a composition series in $\mc O^P$ whose factors are $L(\chi')$ with $\chi'$ strongly linked to $\chi$ and $d\chi' \in \Lambda_{\mf p_L}^+$ (Theorem \ref{theo:constituents}) and for each $\chi'$ there exists a $\mf l_{P,L}$-simple locally analytic representation $U_{\chi'}$ of highest weight $\chi'$ (Proposition \ref{prop:existence-findimirred}).

By Lemma \ref{lemm:gives_old_cond}, $\Hom_{L_P}(U_{\chi'}\otimes_L \pi, J_P(V)) = (0)$ and by adjointness of $J_P(-)$ we deduce $\Hom_{P}(\mc C_c^{\sm}(N_P,U_{\chi'}\otimes_L \pi(\delta_P^{-1})),V) = (0)$. Thus
\begin{multline*}
\Hom_{(\mf g,P)}(L(\chi')\otimes_L \mc C_c^{\sm}(N_P,\pi(\delta_P^{-1})),V)\\ \subset \Hom_{(\mf g,P)}(M_{\mf p_L}(U_{\chi'})\otimes_L \mc C_c^{\sm}(N_P,\pi(\delta_P^{-1})),V)  = (0).
\end{multline*}
Since $L(\chi')$ is an arbitrary irreducible constituent of $M'$ in $\mc O_P$ we deduce $\Hom_{(\mf g,P)}(M''\otimes_L \mc C^{\sm}_c(N_P,\pi(\delta_P^{-1})),V) = (0)$ for $M'' = M'$ or $M'' = (M')^\vee$ (Proposition \ref{prop:irred-objects}(3)). Thus the proposition follows from the exactness properties of $\Hom$.
\end{proof}

The rest of this section is devoted to showing that the map (3a) is surjective. Our method is directly inspired by the proof of \cite[Section 5.6]{be}. Recall that we have fixed our very strongly admissible representation $V$ and our pair $(U,\pi)$ which is non-critical with respect to $V$, and $U$ has the highest weight $\chi$.

\begin{prop}\label{prop:be}
Let $\chi'$ be a locally analytic character of $T$ and let $U_{\chi'}$ be an $L$-linear finite-dimensional $\mrm U(\mf l_{P,L})$-simple representation of $L_P$ with the highest weight $\chi '$. Suppose that $\tilde E$ is an $L$-linear locally analytic representation of $P$ and
\begin{equation}\label{eqn:need-split}
0 \ra V \overset{\alpha}{\longrightarrow} \tilde{E} \overset{\beta}{\longrightarrow} U_{\chi'} \otimes \mc{C} _c ^{\sm}(N_P, \pi(\delta _P ^{-1})) \ra 0
\end{equation}
is an exact sequence of locally analytic representations of $P$. If $\chi'\neq \chi$ is strongly linked to $\chi$ and $d\chi \in \Lambda_{\mf p_L}^+$ then the sequence is split as locally analytic representations of $P$.
\end{prop}
\begin{proof}
The tensor product $U_{\chi'}\otimes_L \mc C_c^{\sm}(N_P,\pi(\delta_{P}^{-1})$ is equipped with its finest convex topology (see the discussion following the statement of Theorem \ref{theo:adj-theo}). Thus \eqref{eqn:need-split} is automatically split in the category of topological spaces. Thus $\tilde E$ is of compact-type (as $V$ is).

On the other hand, since $V$ is f-$\mf{p}$-acyclic, the sequence is split (continuously) as $\mf{p}$-modules (cf.\ Remark \ref{rema-smooth}). If $P^0 \subset P$ is a compact open subgroup then we may integrate a $\mf p$-equivariant splitting to obtain a $P^0$-equivariant splitting.

Let us now twist $\tilde{E}$ by the dual representation $U_{\chi'}^{\ast}$ and define $\tilde{E} '$ as a pullback

\begin{equation*}
\renewcommand{\labelstyle}{\textstyle}
\xymatrix@R=3pc@C=1pc{
0 \ar[r] & U_{\chi'} ^{*} \otimes _L V \ar[r] \ar@{=}[d] & \tilde{E}' \ar[r] \ar@{^{(}->}[d] & L \otimes _L \mc{C} _c ^{\sm}(N_P, \pi(\delta _P ^{-1}))   \ar@{^{(}->}[d] \ar[r] & 0     \\ 
0 \ar[r] & U_{\chi'} ^{*} \otimes _L V  \ar[r] & U_{\chi'} ^{*} \otimes _L \tilde{E} \ar[r]  & \textrm{End} _L(U_{\chi'}) \otimes _L \mc{C} _c ^{\sm}(N_P, \pi(\delta _P ^{-1})) \ar[r] & 0
}
\end{equation*}
so that $\tilde{E}'$ lies in the short exact sequence
\begin{equation}\label{eqn:tilde}
0 \ra U_{\chi'} ^{*} \otimes _L V \ra \tilde{E}' \ra  \mc{C} _c ^{\sm}(N_P, \pi(\delta _P ^{-1}))  \ra 0.  
\end{equation}
By construction it is still a locally analytic representation of $P$ of compact type and the exact sequence (\ref{eqn:tilde}) is split as $P^0$-representations. Moreover, \eqref{eqn:tilde} is split as a sequence of $P$-representations if and only if \eqref{eqn:need-split} is split (because we can reverse the construction and recover $\tilde E$ as a pushout of $\tilde E'$). We now shift our attention to $\tilde{E}'$.

Note that $U_{\chi'}$ is an inflated representation of $L_P$, so $N_P^0=N_P\cap P^0$ acts trivially on $U_{\chi} ^{*}$. Taking $N_P ^0$-invariants in \eqref{eqn:tilde}, the sequence remains exact because \eqref{eqn:tilde} is split as $P^0$-representations. Thus we get an exact sequence
\begin{equation}\label{eqn:tilde2}
0 \ra U_{\chi'} ^{*} \otimes _L V ^{N_P ^0} \ra (\tilde{E}') ^{N_P ^0}\ra  \mc{C} _c ^{\sm}(N_P, \pi(\delta _P ^{-1}))^{N_P ^0}    \ra 0
\end{equation}
of compact-type spaces equipped with continuous actions of the monoid $L_P^+$. By \cite[Proposition 3.4.9]{em2}  and \cite[Lemma 3.5.2]{em2} we have
\begin{equation*}
\mc{C} _c ^{\sm}(N_P, \pi(\delta _P ^{-1}))^{N_P ^0, Z(L_P)^+=\omega _{\pi}} = J_P(\mc{C} _c ^{\sm}(N_P, \pi(\delta _P ^{-1})))^{Z(L_P)=\omega _{\pi}} = \pi.
\end{equation*}
Since $(U,\pi)$ is non-critical with respect to $V$ and $\chi'\neq \chi$ is strongly linked to $\chi$ with $d\chi' \in \Lambda_{\mf p_L}^+$ we also know that $(U_{\chi'}^{\ast}\otimes_L V^{N^0_P})^{Z(L_P)^+=\omega_\pi} = (0)$. Since taking eigenspaces is left exact we deduce
\begin{equation}\label{eqn:want-iso-pi}
(\tilde{E}')^{N^0_P,Z(L_P)^+=\omega_\pi} \hookrightarrow \pi.
\end{equation}
We equip the left-hand side of \eqref{eqn:want-iso-pi} with an action of $Z(L_P)$ via $\omega_\pi$. This is compatible with the action of $L_P^+$ in that it evidently agrees on the intersection $Z(L_P)^+ = Z(L_P) \cap L_P^+$. Thus we get an action of $L_P$ on the left-hand side of \eqref{eqn:want-iso-pi} (see \cite[Proposition 3.3.6]{em2} and compare with \cite[Proposition 3.4.9]{em2}) and \eqref{eqn:want-iso-pi} is equivariant for the $L_P$-actions. We claim that it is an isomorphism. Granting that, we deduce an $L_P$-equivariant inclusion $\pi \hookrightarrow J_P(\tilde{E}')$ and by adjunction this gives a $P$-equivariant morphism $\mc C^{\sm}_c(N_P,\pi(\delta_P^{-1})) \rightarrow \tilde E'$ which one easily checks is a section to the quotient in \eqref{eqn:tilde}.

Thus we have reduced to showing \eqref{eqn:want-iso-pi} is an isomorphism, and we know that it is injective. Since $\pi$ is smooth, it suffices to show $\pi^H$ is in the image of \eqref{eqn:want-iso-pi} for all compact open $H \subset L_P$. If $H' \subset H$ then $\pi^H \subset \pi^{H'}$ and so without loss of generality we may suppose that $H \subset L_P^0 \subset L_P^+$. 

The final term of \eqref{eqn:tilde2} canonically contains $\pi$ (as we saw above). Write $F \subset (\tilde E')^{N^0_P}$ for the preimage of $\pi$ so $F$ surjects onto $\pi$. Since $H \subset P^0$, the sequence \eqref{eqn:tilde2} induces an exact sequence
\begin{equation}\label{eqn:Hsequence}
0 \ra \left(U_{\chi'}\otimes_L V^{N^0_P}\right)^H \ra F^H \ra \mc \pi^H \ra 0.
\end{equation}
The terms in sequence \eqref{eqn:Hsequence} are equipped with continuous actions of $Z(L_P)^+$. 

Since $V$ is an admissible locally analytic representation of $G$, there exists a $z \in Z(L_P)^+$ so that the operator $\pi_z$ acts on $(U_{\chi'}\otimes_L V^{N^0_P})^H$ compactly (this is contained in the proofs of \cite[Proposition 4.2.33]{em2} or \cite[Theorem 4.10]{hl}, namely the parts of those proofs which do not mention the $J_P(V)$; compare with the proof of  \cite[Proposition 5.5.4]{be}). On the other hand, $\pi^H$ is finite-dimensional since $\pi$ is admissible. Thus $\pi_z$ also acts compactly on $F^H$. We deduce that taking generalized $Z(L_P)^+$-eigenspaces on \eqref{eqn:Hsequence} is exact, and that generalized $Z(L_P)^+$-eigenspaces vanish if and only if bona fide $Z(L_P)^+$-eigenspaces vanish.

But now the $Z(L_P)^+$-eigenspace for $\omega_\pi$ is trivial on the left-hand side of \eqref{eqn:Hsequence} and thus $(F^H)^{(Z(L_P)^+=\omega_\pi)} \simeq (\pi^H)^{(Z(L_P)^+=\omega_\pi)}$ (the parentheses indicate generalized eigenspaces). Finally, $Z(L_P)^+$ acts on all $\pi^H$ by $\omega_\pi$ and so we deduce $F^{H,Z(L_P)^+=\omega_\pi} \simeq \pi^H$. But $F \subset (\tilde{E}')^{N^0_P}$ and so we have shown that $\pi^H$ is contained in the image of \eqref{eqn:want-iso-pi} which completes the proof of the claim.
\end{proof}

\begin{theo}\label{theo:sur}
Suppose $\chi'\neq \chi$ is a locally analytic character of $T$ which is strongly linked to $\chi$, and $M \subset M'$ are two objects of $\mc{O}^P$ such that $M' / M \simeq L(\chi')$. Then the natural map 
\begin{equation*}
\Hom _{(\mf{g},P)}(M' \otimes _L \mc{C} _c ^{\sm}(N_P, \pi(\delta _P ^{-1})), V) \ra \Hom _{(\mf{g},P)}(M \otimes _L \mc{C} _c ^{\sm}(N_P, \pi(\delta _P ^{-1})), V)
\end{equation*}
is surjective.
\end{theo}
\begin{proof}
Since $M'/M = L(\chi')$ is in $\mc O^P$, $d\chi' \in \Lambda_{\mf p_L}^+$ (Lemma \ref{lemma:L(chi)inOP}) and we have a short exact sequence of $(\mf g,P)$-modules 
\begin{multline}\label{eqn:need-to-push}
0 \ra M\otimes_L \mc C_c^{\sm}(N_P,\pi(\delta_P^{-1})) \ra M'\otimes_L \mc C_c^{\sm}(N_P,\pi(\delta_P^{-1}))\\ \ra L(\chi')\otimes_L \mc C_c^{\sm}(N_P,\pi(\delta_P^{-1})) \ra 0.
\end{multline}
Let $f: M\otimes_L \mc C_c^{\sm}(N_P,\pi(\delta_P^{-1})) \ra V$ be $(\mf g,P)$-equivariant, and then form the pushout $E$ of the sequence \eqref{eqn:need-to-push} along $f$, so that $E$ is naturally a locally analytic $(\mf g,P)$-module sitting in a diagram
\begin{equation*}
\xymatrix{
0 \ar[d] & 0 \ar[d]\\
M\otimes_L \mc C_c^{\sm}(N_P,\pi(\delta_P^{-1})) \ar[d] \ar[r]^-{f} & V \ar[d]^-{\alpha_0}\\
M'\otimes_L \mc C_c^{\sm}(N_P,\pi(\delta_P^{-1})) \ar[d] \ar[r]^-{f'} & E \ar[d]^-{\beta_0} \\
L(\chi')\otimes_L \mc C_c^{\sm}(N_P,\pi(\delta_P^{-1})) \ar[d] \ar@{=}[r] &L(\chi')\otimes_L \mc C_c^{\sm}(N_P,\pi(\delta_P^{-1}))   \ar[d] \\
0 & 0
}
\end{equation*} 
To prove the proposition, it is enough to construct a section of $\beta_0$. Indeed, if $s_0$ is the corresponding section of $\alpha_0$ then $s_0\circ f'$ is an extension of $f$ to $M'\otimes_L \mc C_c^{\sm}(N_P,\pi(\delta_P^{-1})$.

Because $L(\chi ') \in \mc{O} ^P$, by Proposition \ref{prop:existence-findimirred} we get a locally analytic representation $U_{\chi'}$ of the highest weight $\chi'$. We can take the right-hand vertical sequence and pull it back along the quotient map $M_{\mf p _L}(U_{\chi'}) \twoheadrightarrow L(\chi')$. Then, we get a locally analytic $(\mf g,P)$-module $\tilde E_0$ defined by the diagram
\begin{equation}\label{eqn:split-diagram}
\xymatrix{
0 \ar[d] & 0 \ar[d]\\
V \ar[d]_-{\alpha_1} \ar@{=}[r] & V \ar[d]^-{\alpha_0}\\
\tilde E_0 \ar[d]_-{\beta_1} \ar[r]^-{f''} & E \ar[d]^-{\beta_0} \\
M_{\mf p _L}(U_{\chi'})\otimes_L \mc C_c^{\sm}(N_P,\pi(\delta_P^{-1})) \ar[d] \ar[r] &L(\chi')\otimes_L \mc C_c^{\sm}(N_P,\pi(\delta_P^{-1}))   \ar[d] \\
0 & 0
}
\end{equation}
We claim that to split the right-hand vertical sequence, it is enough to split the left-hand vertical sequence. To see that, suppose that $s_1$ is a section of $\beta_1$. Then we get a morphism 
\begin{equation*}
f''\circ s_1: M_{\mf p _L}(U_{\chi'})\otimes_L \mc C_c^{\sm}(N_P,\pi(\delta_P^{-1})) \ra E.
\end{equation*}
The non-criticality of $(U,\pi)$ with respect to $V$ implies that $f''\circ s_1$ must factor through $L(\chi')\otimes_L \mc C_c^{\sm}(N_P,\pi(\delta_P^{-1}))$, and thus defines a section of $\beta_0$. Indeed, if we write $M = \ker\left(M_{\mf p _L}(U_{\chi'}) \ra L(\chi')\right)$ then $f''\circ s_1$ induces a map
\begin{equation}\label{eqn:must-be-zero}
f'' \circ s_1 : M\otimes_L \mc C_c^{\sm}(N_P,\pi(\delta_P^{-1})) \ra V.
\end{equation}
The constituents of $M$ are all of the form $L(\chi'')$ with $\chi''$ strongly linked to $\chi'$, and thus also strongly linked (and not equal) to $\chi$. Moreover each $\chi''$ has $d\chi'' \in \Lambda_{\mf p_L}^+$ (Lemma \ref{lemma:L(chi)inOP}). By Lemma \ref{lemm:gives_old_cond}, and the right-hand side of the main diagram on page \pageref{eqn:big_diagram}, we deduce that every map $L(\chi'') \otimes_L \mc C_c^{\sm}(N_P,\pi(\delta_P^{-1})) \ra V$ is zero, and thus \eqref{eqn:must-be-zero} is also zero by d\'evissage.

We have now reduced the theorem to showing the left-hand vertical sequence in \eqref{eqn:split-diagram} is split. But, if we let $\tilde E \subset \tilde E_0$ be the $P$-stable subspace which is the preimage of $U_{\chi'}\otimes_L \mc C_c^{\sm}(N_P,\pi(\delta_P^{-1}))$ under the map $\beta_1$ then we get an exact sequence
\begin{equation*}
0 \ra V \ra \tilde E \ra U_{\chi'} \otimes_L \mc C_c^{\sm}(N_P,\pi(\delta_P^{-1})) \ra 0
\end{equation*}
of locally analytic representations of $P$. By Proposition \ref{prop:be} it is split and any $P$-equivariant splitting induces a $(\mf g,P)$-equivariant splitting of $\beta_1$ by extending $\mrm U(\mf g _L)$-equivariantly (the result is easily checked to be $P$-equivariant still). This completes the proof.
\end{proof}
\begin{rema}
The middle step of the previous proof, where we considered the constituents of $M_{\mf p _L}(U_{\chi'})$, is where we need to know that $J_P(V)^{Z(L_P)=\omega_{\pi}\chi''} = (0)$ not just for constituents $L(\chi'')$ of $M_{\mf p _L}(U)$ but for all strongly linked characters $\chi'' \neq \chi$ with $L(\chi'')$ appearing in $\mathcal O^P$. (Compare with Remark \ref{remark:we-wish}.)
\end{rema}

We can finally finish the proof of the main theorem.
\begin{coro}\label{coro:final-corollary}
The map (3a) is surjective.
\end{coro}
\begin{proof}
If $M \subset M_{\mf p _L}(U)^{\vee}$ is a subobject in $\mc O^P$ then the constituents of $M_{\mf p _L}(U)^{\vee}/M$ are all of the form $L(\chi')$ with $\chi'\neq\chi$ strongly linked to $\chi'$ and $d\chi' \in \Lambda_{\mf p_L}^+$ (see Proposition \ref{prop:irred-objects} and Theorem \ref{theo:constituents}). Thus the corollary follows from Theorem \ref{theo:sur} by descending induction on the length of $M$, starting with $M = M_{\mf p _L}(U)^{\vee}$ and ending with $M = L(\chi)$.
\end{proof}

\section{Restriction to the socle}\label{section:socle}
The goal of this section to improve Theorem \ref{theo:adj-theo}. Namely, we study the restriction of morphisms $\Ind_{P^-}^G(U\otimes \pi(\delta_P^{-1})) \ra V$ which arise via adjunction in the non-critical case to the {\em locally analytic socle}. Unlike the rest of this article, we restrict here to the case of finite-dimensional irreducible $\mbf Q_p$-algebraic representations $U$ of $L_P$, i.e.\ the canonical action of $L_P$ on a finite-dimensional irreducible algebraic representation of the underlying reductive group $\Res_{K/\mbf Q_p} \mbf L_{\mbf P}$. This is necessary to use of results of Breuil in \cite{br}, extending recent progress of Orlik--Strauch \cite{os-old, os} on studying the locally analytic principal series.

\subsection{The Orlik--Strauch representations}
Suppose that $M \in \mc O^P$ and $\pi$ is a smooth admissible representation of $L_P$. In \cite{os} (and \cite{os-old}) Orlik and Strauch  define a locally analytic representation $\mc F_{P}^G(M,\pi)$ of the group $G$. The association $(M,\pi) \mapsto \mc F_P^{G}(M,\pi)$ is a functor, contravariant in $M$ and covariant in $\pi$. It is exact in both arguments (see \cite[Section 3]{os} for these results). When $U$ is in $\mc O^{L_P}$, we can consider the generalized Verma module $M_{\mf p _L}(U) \in \mc O^P$. The definition of the Orlik--Strauch representations immediately gives
\begin{equation*}
\mc F_P^G(M_{\mf p _L}(U), \pi) = \Ind_P^G(U^{\ast} \otimes \pi)^{\an},
\end{equation*}
so these representations naturally include locally analytic principal series (note that the representation $U^{\ast}$ on the right is the {\em dual} of the representation $U$ on the left).

Recall that if $M \in \mc O^P$ then we say that $\mf p_L$ is {\em maximal} for $M$ if $M \not\in \mc O^{\mf q_L}$ for all $\mf q_L \supsetneq \mf p_L$. Since $\mc F_P^G(M,\pi)$ is exact in both arguments there are obvious necessary conditions for the irreducibility. Orlik and Strauch established sufficient conditions.
\begin{theo}[{\cite[Theorem 1.1]{os}}]\label{theo:irreducibility}
If $M \in \mc O^P$ is simple, $\mf p$ is maximal for $M$ and $\pi$ is an irreducible smooth representation of $L_P$ then $\mc F_P^G(M,\pi)$ is topologically irreducible. 
\end{theo}
\begin{rema}\label{rem:pq-formula}
It is possible to deduce the irreducibility in some cases without assuming that $P$ is the maximal parabolic. This is done by using a relation (the ``$PQ$''-formula) between $\mc F_Q^G(-,-)$ and $\mc F_P^G(-,-)$ for a containment $P \subset Q$ of a parabolics. See the proof Theorem \ref{coro:socle}.
\end{rema}

\subsection{An adjunction formula with the socle in the algebraic case}
Prior to the work \cite{os}, results mentioned in the previous section (especially Theorem \ref{theo:irreducibility}) were established in the algebraic case (see \cite[Theorem 5.8]{os-old} for example). We recall the algebraic case now and prove Theorem B.

Following \cite{os-old}, we write $\mc{O} _{\alg}$ for the subcategory of $M \in \mc O^{\mf b_L}$ such that all the weights of $\mf t_L$ acting on $M$ are restrictions of $\mbf Q_p$-algebraic characters of $T$. Such $M$ are generated over $\mrm U(\mf g _L)$ by highest weight vectors of algebraic weight. We also write $\mc{O} ^{\mf p_L} _{\alg} := \mc{O} _{\alg} \cap \mc{O} ^{\mf p_L}$. The $\mrm U(\mf g _L)$-module structure on an element in $\mc O_{\alg}^{\mf p_L}$ naturally extends to a locally analytic action of $P$, defining a fully faithful embedding $\mc O_{\alg}^{\mf p_L} \hookrightarrow \mc O^P$ (see \cite[Example 2.4(i)]{os}). If $U$ is an $L$-linear finite-dimensional locally analytic representation of $L_P$, the generalized Verma module $M_{\mf p_L}(U)$ will lie in $\mc O^{\mf p_L}_{\alg}$ if and only if $U$ is an $\mbf Q_p$-algebraic representation of $L_P$.

Within the category $\mc O_{\alg}$, Breuil was able to prove an adjunction formula for the Emerton--Jacquet functor (see \cite[Th\'eor\`eme 4.3]{br}) in which, rather than locally analytic principal series appearing, one has the Orlik--Strauch representations appearing. The key computation is contained in the following proposition which we will use. 

Recall that if $M \in \mc O_{\alg}^{\mf p_L^-} \subset \mc O^{P^-}$ then it has an opposite dual $M^{-} \in \mc O_{\alg}^{\mf p} \subset \mc O^P$.

\begin{prop}[{\cite[Proposition 4.2]{br}}]\label{br-adj}
Suppose that $M\in \mc{O} _{\alg} ^{\mf p_L ^-}$ and $\pi$ is a smooth admissible representation of $L_P$ of finite length. Let $V$ be an $L$-linear very strongly admissible representation of $G$. Then there is a canonical isomorphism
\begin{equation*}
\Hom _G (\mc{F}_{P^-} ^G(M,\pi), V) \simeq \Hom _{(\mf g, P)}(M ^{-} \otimes \mc{C} _c ^{\sm}(N_P, \pi), V).
\end{equation*}
\end{prop}

Our improvement of Theorem \ref{theo:adj-theo} is a combination of the previous two results. If $U$ is an irreducible, finite-dimensional, algebraic representation of $L_P$ then it is $\mf l_P$-simple and thus the results of Section \ref{sec:adj-formula} apply (see Example \ref{exam:include_loc_algebraics}). If $\chi$ is the algebraic highest weight character of $U$ then $\chi^{-1}$ is the highest weight for the dual representation $U^{\ast}$. In particular, if $Q$ is the maximal parabolic for $U$ then $Q^-$ is the maximal parabolic for $U^{\ast}$.

\begin{theo}\label{coro:socle}
Let $U$ be an irreducible finite-dimensional algebraic representation of $L_P$ with maximal parabolic $Q$ and suppose that $\pi$ is a finite length smooth admissible representation of $L_P$ admitting a central character such that $\Ind_{P^- \cap L_Q}^{L_Q}(\pi)^{\sm}$ is irreducible. Let $V$ be an $L$-linear very strongly admissible and f-$\mf p$-acyclic representation of $G$ such that $(U,\pi(\delta_P))$ is non-critical with respect to $V$. Then the containment
\begin{equation*}
\soc _G \Ind _{P^-} ^G(U \otimes \pi) ^{\an} \subset \Ind _{P^-} ^G(U \otimes \pi) ^{\an}
\end{equation*}
defines a natural isomorphism
\begin{equation*}
\Hom _G(\Ind _{P^-} ^G(U \otimes \pi) ^{\an}, V) \simeq \Hom _G (\soc _G \Ind _{P^-} ^G(U \otimes \pi) ^{\an}, V).
\end{equation*}
\end{theo}
Note that the non-critical hypothesis is with respect to the pair $(U,\pi(\delta_P))$. This is due to our normalizations (see Section \ref{subsec:notations}) of the Emerton--Jacquet functor and to help remove the twists by $\delta_P^{-1}$ that would appear in each line of the following proof. We comment on the rest of the hypotheses in Theorem \ref{coro:socle} after the proof.
\begin{proof}[Proof of Theorem \ref{coro:socle}]
Write $\chi$ for the highest weight of $U$. Consider the irreducible object $L(\chi^{-1}) \in \mc O^{P^-}$. By \cite[Proposition 3.12(b)]{os} we have
\begin{equation*}
\mc F_{P^-}^G(L(\chi^{-1}),\pi)  = \mc F_{Q^-}^G(L(\chi^{-1}), \Ind_{P^- \cap L_Q}^{L_Q}(\pi)^{\sm})
\end{equation*}
(this is the ``$PQ$''-formula mentioned in Remark \ref{rem:pq-formula}). Since $\Ind_{P^-\cap L_Q}^{L_Q}(\pi)^{\sm}$ is irreducible, Theorem \ref{theo:irreducibility} implies that the right hand side is irreducible.

On the other hand, the exactness of $\mc F_{P^-}^G(-,\pi)$ and the relation between the Orlik--Strauch representations and parabolic induction implies that
\begin{equation*}
\soc_G \Ind_{P^-}^G(U\otimes \pi)^{\an} = \soc_G \mc F_{P^-}^G(M_{\mf p^- _L}(U^{\ast}),\pi) = \soc_G \mc F_{P^-}^G(L(\chi^{-1}),\pi)
\end{equation*}
As the last term is irreducible, we deduce that we have a natural commuting diagram
\begin{equation}\label{eqn:f}
\xymatrix{
\mc{F} _{P^-} ^G(L(\chi ^{-1}), \pi) \ar[d]_-{\simeq} \ar[r] &  \mc{F} _{P^-} ^G(M_{\mf p ^- _L}(U^{\ast}), \pi) \ar[d]^-{\simeq} \\
\soc_G \Ind_{P^-}^G(U \otimes \pi)^{\an} \ar[r] & \Ind_{P^-}^G(U \otimes \pi)^{\an}
}
\end{equation}
where the horizontal arrows are inclusions (this also shows that our assumptions force the locally analytic socle to be irreducible).

By Proposition \ref{prop:poly-internal-dual} (cf.\ its proof) we have a natural isomorphism $M_{\mf p^- _L}(U^{\ast})^- \simeq M_{\mf p _L}(U)^{\vee}$. This also implies that there is a natural isomorphism $L(\chi^{-1})^- \simeq L(\chi)^\vee$ (after taking internal duals both are quotients of $M_{\mf p _L}(U)$) and Proposition \ref{prop:irred-objects}(c) gives a natural isomorphism $L(\chi)^\vee \simeq L(\chi)$. Finally, Proposition \ref{br-adj} implies that the quotient map $M_{\mf p^- _L}(U^{\ast}) \twoheadrightarrow L(\chi^{-1})$ induces a canonical commuting diagram
\begin{equation*}
\xymatrix{
\Hom_G\left(\mc F_{P^-}^G(M_{\mf p^- _L}(U^{\ast}),\pi), V \right)\ar[d] \ar[r]^-{\simeq} & \Hom_{(\mf g,P)}\left(M_{\mf p _L}(U)^{\vee} \otimes \mc C_c^{\sm}(N_P,\pi), V\right) \ar[d]^-{\text{(3a)}}\\
\Hom_G\left(\mc F_{P^-}^G(L(\chi^{-1}),\pi), V \right) \ar[r]^-{\simeq} & \Hom_{(\mf g,P)}\left(L(\chi) \otimes \mc C_c^{\sm}(N_P,\pi), V\right)
}
\end{equation*}
Here we have labeled the right hand vertical arrow as (3a), as that is the same map labeled (3a) in the diagram on page \pageref{eqn:big_diagram}. Recall we have proved that if $(U,\pi(\delta_P))$ is non-critical with respect to $V$ then the map (3a) is an isomorphism; the injectivity was proven in Proposition \ref{prop:iso-and-inj} and surjectivity in Corollary \ref{coro:final-corollary}. We conclude that the restriction map 
\begin{equation*}
\Hom_G\left(\mc F_{P^-}^G(M_{\mf p^- _L}(U^{\ast}),\pi), V \right) \ra \Hom_G\left(\mc F_{P^-}^G(L(\chi^{-1}),\pi), V \right)
\end{equation*}
is an isomorphism. By the diagram \eqref{eqn:f}, the restriction  map
\begin{equation*}
\Hom_G(\Ind_{P^-}^G(U\otimes \pi)^{\an}, V) \ra \Hom_G(\soc_G \Ind_{P^-}^G(U\otimes \pi)^{\an}, V)
\end{equation*}
is also an isomorphism. This concludes the proof.
\end{proof}

\begin{rema}\label{rema:irreducible}
The algebraic assumption on $U$ in Theorem \ref{coro:socle} could be removed, as long as the $\mf l_P$-simple assumption is kept, with a suitable generalization of Proposition \ref{br-adj} to the category $\mc O^P$.
\end{rema}

\begin{rema}\label{rema:irreducibility}
Regarding the hypothesis on $\pi$, it is sufficient, but not necessary, to assume that $\Ind_{P^-}^G(\pi)^{\sm}$ is irreducible. If $\mbf G = {\GL_n}_{/K}$ then a well-known criterion comes out of the Bernstein--Zelevinsky classification (see \cite[Theorem 4.2]{bz}). For example, if $\pi =\theta_1 \otimes \dotsb \otimes \theta_n$ is a smooth character of the diagonal torus $T$ then it is necessary and sufficient to assume that $\theta_i(\varpi_K)/\theta_j(\varpi_K) \neq q$ if $i\neq j$, where $\varpi_K$ is any uniformizer of $K$ and the residue field of $K$ has $q$ elements.
\end{rema}

\end{document}